\documentclass[11 pt, leqno]{amsart}
\usepackage{ragged2e}
\usepackage{color}
\usepackage{a4wide}
\usepackage[utf8]{inputenc}
\usepackage{enumerate}
\usepackage{amsmath}
\usepackage{amsfonts}
\usepackage{amssymb}
\usepackage{mathtools}
\usepackage{tikz-cd}
\usepackage{bbm}
\usepackage{mathrsfs}
\usepackage{amsthm}
\usepackage{verbatim}
\usepackage{upgreek}
\usepackage{relsize}
\usepackage{enumitem}
\usepackage{dsfont}
\usepackage{graphicx}
\usepackage{amssymb}
\usepackage[all]{xy}
\usepackage{mathabx}%otop
\usepackage{url}

\newcommand{\eps}{\varepsilon}
\newcommand{\ov}{\overline}
\newcommand{\id}{\textnormal{id}}
\newcommand{\mc}{\mathcal}

\newcommand{\mrm}{\mathrm}

\newcommand{\mf}{\mathfrak}
\newcommand{\msf}{\mathsf}
\newcommand{\I}{\mathbbm{1}}

\newcommand{\md}{\operatorname{d}\!}
\newcommand{\cst}{\ifmmode \mathrm{C}^* \else $\mathrm{C}^*$\fi}

\newcommand{\la}{\langle}
\newcommand{\ra}{\rangle}

\newcommand{\bbGamma}{{\mathpalette\makebbGamma\relax}}
\newcommand{\makebbGamma}[2]{%
  \raisebox{\depth}{\scalebox{1}[-1]{$\mathsurround=0pt#1\mathbb{L}$}}%
}

\newcommand{\NN}{\mathbb{N}}
\newcommand{\RR}{\mathbb{R}}

\newcommand{\CC}{\mathbb{C}}
\newcommand{\ZZ}{\mathbb{Z}}
\newcommand{\GG}{\mathbb{G}}

\newcommand{\cC}{\mathcal{C}}

\newcommand{\wot}{\ifmmode \textsc{wot} \else \textsc{wot}\fi}
\newcommand{\sot}{\ifmmode \textsc{sot} \else \textsc{sot}\fi}
\newcommand{\sots}{\ifmmode \textsc{sot}^* \else \textsc{sot}$^*$\fi}
\newcommand{\ssot}{\ifmmode \sigma\textsc{-sot} \else $\sigma$-\textsc{sot }\fi}
\newcommand{\ssots}{\ifmmode \sigma\textsc{-sot}^* \else $\sigma$-\textsc{sot }$^*$\fi}
\newcommand{\swot}{\ifmmode \sigma\textsc{-wot} \else $\sigma$-\textsc{wot}\fi}

\newcommand{\wh}{\widehat}
\newcommand{\whG}{\widehat{\GG}}

\newcommand{\oon}{\operatorname}

\renewcommand{\restriction}{\mathord{\upharpoonright}}

\DeclareMathOperator{\Fol}{F\mathrm{\o}l}
\DeclareMathOperator{\Folinn}{F\mathrm{\o}l^{inn}}
\newcommand{\FolinnX}{\operatorname{F\mathrm{\o}l}_{X}^{\operatorname{inn}}}
\newcommand{\Folinnj}{\operatorname{F\mathrm{\o}l}_{\{1\}}^{\operatorname{inn}}}

\DeclareMathOperator{\Kaz}{K}

\DeclareMathOperator{\lin}{span}
\DeclareMathOperator{\Irr}{Irr}
\DeclareMathOperator{\Pol}{Pol}
\DeclareMathOperator{\supp}{supp}
\DeclareMathOperator{\Tr}{Tr}
\DeclareMathOperator{\B}{B}

\DeclareMathOperator{\M}{M}
\DeclareMathOperator{\N}{N}

\DeclareMathOperator{\Rep}{Rep}
\DeclareMathOperator{\SU}{SU}

\DeclareMathOperator{\inn}{inn}

\newtheorem{theorem}{Theorem}[section]

\newtheorem{proposition}[theorem]{Proposition}
\newtheorem{lemma}[theorem]{Lemma}
\theoremstyle{definition}
\newtheorem{corollary}[theorem]{Corollary}
\newtheorem{remark}[theorem]{Remark}

\newtheorem{definition}[theorem]{Definition}

\numberwithin{equation}{section}
\newtheorem{deft}[theorem]{Definition}

\begin{document}
		
\renewcommand{\arraystretch}{2}
		
\author{Jacek Krajczok}
\address{Vrije Universiteit Brussel, Pleinlaan 2, 1050 Brussels, Belgium}
\email{jacek.krajczok@vub.be}

\author{Adam Skalski}
\address{Institute of Mathematics of the Polish Academy of Sciences, ul. \'Sniadeckich 8, 00-656 Warszawa, Poland}
\email{a.skalski@impan.pl}

%\author{Jacek Krajczok}
%\address{Vrije Universiteit Brussel, Pleinlaan 2, 1050 Brussels, Belgium}
%\email{jacek.krajczok@vub.be}

%\author{Adam Skalski}
%\address{Institute of Mathematics of the Polish Academy of Sciences, ul. \'Sniadeckich 8, 00-656 Warszawa, Poland}
%\email{a.skalski@impan.pl}

\title[Asymptotic invariants of quantum fusion algebras]{\bf Asymptotic invariants for fusion algebras associated with compact quantum groups}

\begin{abstract}
We introduce and study certain asymptotic invariants associated with fusion algebras (equipped with a dimension function), which arise naturally in the representation theory of compact quantum groups. Our invariants generalise the analogous concepts studied for classical discrete groups. Specifically, we introduce uniform F\o lner constants and the uniform Kazhdan constant for a regular representation of a fusion algebra, and establish a relationship between these, amenability, and the exponential growth rate considered earlier by Banica and Vergnioux. Further we compute the invariants for fusion algebras associated with 
% discrete duals of
quantum $SU_q(2)$ and $SO_q(3)$ and determine the uniform exponential growth rate for the fusion algebras of all $q$-deformations of semisimple, simply connected, compact Lie groups and for all free unitary quantum groups.
\end{abstract}

\subjclass[2020]{Primary 46L65; Secondary 18M20, 20F69, 20G42,  46L89}
\keywords{Fusion algebra; uniform non-amenability; compact quantum group; Kazhdan constants; uniform isoperimetric constants; growth rates}

\maketitle
\section{Introduction}
Recent years have brought an increased interest in extending classical geometric group theory notions to the case of locally compact -- and especially discrete -- quantum groups. A notable example is given by \emph{amenability}, a central concept of group theory, which plays a similarly important role in the quantum context. It was relatively early understood by Banica that amenability of a given discrete quantum group $\bbGamma$ can be described in terms of the properties of the fusion algebra encoding the representation theory of its compact dual $\GG$ (\cite[Theorem 6.2]{BanicaCrelle}). Here one can think about elements of the relevant fusion algebra as equivalence classes of finite-dimensional representations of $\GG$, equipped with the tensor product and a natural quantum dimension function. At the same time Hiai and Izumi, inspired by the subfactor theory, undertook in \cite{HiaiIzumi} a deep study of amenability for fusion algebras, which itself led to later work by Kyed on F\o lner sets and $\ell^2$-Betti numbers for discrete quantum groups (see \cite{Kyed} and references therein).

Since then, there have been many applications of properties of fusion algebras, and more generally the related rigid tensor categories, to the study of qualitative geometric and analytic aspects of discrete quantum groups, such as amenability or Property (T). Note for example \cite{PopaVaes}, \cite{SergeyMakoto}, \cite{KRVV} or the book \cite{NeshveyevTuset}. On the other hand, when studying classical discrete groups, one is often interested in finer, preferably quantitative asymptotic invariants. An old outstanding example is the \emph{rate of  growth} for a given finitely generated group.  When the group in question is amenable, one may ask about its \emph{isoperimetric profile} (see \cite{Erschler} and references therein). On the other hand, for non-amenable groups  the articles  \cite{Arzhantseva}, \cite{LPV}, \cite{Shalom} introduce and study asymptotic invariants which in a sense quantify the `degree of non-amenability'.

In this work, inspired by the articles quoted above, we introduce and investigate analogous notions (namely \emph{uniform F\o lner constants}   $\Folinn(R)$ and $\Fol(R)$, and the \emph{uniform Kazhdan constant for the regular representation} $\Kaz(\rho, R)$) for an arbitrary fusion algebra $R$. Whilst our motivations come from quantum groups, the invariants introduced can be used as quantifiers of non-amenability for arbitrary fusion algebras. As in the classical case, we connect the notions to the concept of the \emph{uniform exponential growth rate} $\omega(R)$, whose non-uniform version was introduced for discrete quantum groups in \cite{BanicaVergnioux}. The basic relations between the constants above are given by the following inequalities:
\[ 
\Folinn(R)\le \Fol(R),\quad \Kaz(\rho,R)^2\le 
 \Fol(R) ,\quad \Folinn(R)\le 1-\tfrac{1}{\omega(R)}. 
\]
Moreover, if $R$ is amenable, then $\Folinn(R) = \Fol(R) = \Kaz(\rho,R) =0$. 

Having established these general properties, we proceed to compute the invariants for several concrete examples of fusion algebras, focusing on the case of fusion algebras of  $q$-deformations of classical compact Lie groups. We obtain complete results for the fusion algebras associated with quantum $SU_q(2)$ and $SO_q(3)$, exploiting a relatively simple form of the related fusion rules and applying classical analytic techniques. For $q$-deformations of general semisimple, simply connected\footnote{We follow the convention in which every simply connected space is path connected.}, compact Lie groups the combinatorial problems become formidable, but we nevertheless compute the uniform exponential growth rates, using fundamental aspects of the representation theory of Lie groups and explicit computation with Sage \cite{sagemath}. Similarly, we compute the uniform exponential growth rate for the fusion algebras associated with the free unitary quantum groups $U_F^+$ of Van Daele and Wang.

The table below summarises the main computational results of our work; for comparison we also add corresponding  statements regarding amenable fusion algebras.

\begin{table}[ht]
	\begin{center}	
		\begin{tabular}{|{c} |{c}|{c}|{c}| {c}| }  
			\hline
			Fusion algebra & F\o l(R) &  $\Folinn$(R) & $\omega(R)$ & $\Kaz(\rho, R)$\\	\hline
			amenable &  0  & 0 & no information  &  0 \\ 
			$\;\;R(SU_q(2)) = R(O_F^+)\;\;$ & $q^{-2} - q^{2}$ &  $1-q^2$  & $q^{-2}$   &$\frac{(1-q)^2}{q^2+1}$\\
			$R(SO_q(3)) = R(S_N^+)$ & $q^{-4} - q^4$ &  $1-q^4$  & $q^{-4}$   & $1- \tfrac{3}{q^2 +q^{-2}+1}$\\
			$R(G_q)$ &  ??&  ??  & $\mathcal{F}(G,q)$   &  ??\\
			$R(U_F^+)$ &  ??&  ??  & $r_q$   &  ??\\
			\hline
		\end{tabular}
	\end{center}
\end{table}
Note that for $R(SU_q(2)) = R(O_F^+)$ and $R(SO_q(3))=R(S_N^+)$ we have an explicit relationship between  the matrix $F$ (or the number $N$) and parameter $q$. The function $\mathcal{F}(G,q)$ depends in particular on the type of the compact Lie group in question; and $r_q$ is the largest real root of a certain explicit cubic polynomial depending on the parameter $q\in(0,1]$ determined by the matrix $F\in GL_N( \CC)$. We obtain also precise asymptotics of $r_q$ with $q$ tending to $0^+$. The missing constants concerning the $R(G_q)$ in the table above could in principle be computed case by case, although the combinatorics appears formidable; the case of $U_F^+$ appears more challenging conceptually, and relates to the study of weighted trees, as discussed below.

Our article should be viewed only as a starting point of the study of extensions of numerical asymptotic invariants 
of discrete groups to the context of fusion algebras. We would like to end this introduction by mentioning two  natural open questions.

Firstly, it is natural to expect that the uniform F\o lner constants of a fusion algebra of a quantum group are also related to the first \emph{$\ell^2$-Betti number} of the associated tensor category (as studied for example in \cite{Kyed} and \cite{KRVV}). In the classical context such a relation -- in fact exploiting a yet different variant of the constants studied in this paper, counting `boundary edges' rather than `boundary points'  --  was established in \cite{LPV}. We expect that we should also see a similar phenomenon for general fusion algebras.

Secondly, as recorded in the table above, in Section \ref{Sec:SU} we compute all the invariants for fusion algebras associated to quantum $SU_q(2)$, which can be also viewed as associated to \emph{free orthogonal quantum groups} $O_F^+$ of Van Daele and Wang (\cite{VDW}), see \cite{BanicaOrth}. Thus, the next step would be to do the same for \emph{free unitary quantum groups} $U^+_F$. Here the related geometric structure describing the fusion rules is a tree -- see 
\cite{BanicaUnitary} -- which we exploit in this work to compute the uniform exponential growth rate. The problem of computing the uniform F\o lner constants is naturally connected to the analogous task for classical free groups, studied in \cite{Arzhantseva}. However, we were so far not been able to determine the corresponding constants, as the `representation tree' of $U_F^+$ is \emph{weighted}, and we cannot use the simple Euler characteristic method as in \cite[Proposition 5.2]{Arzhantseva}.

\smallskip

The detailed plan of the paper is as follows: in Section \ref{sec:fusiondef} we introduce basic definitions and facts concerning fusion algebras and explain how these arise in the context of representation theories of compact quantum groups. Here we also recall the notion of $q$-numbers and establish some inequalities which will be used later. Section \ref{sec:asymptotic} introduces key abstract results of the paper, defining our asymptotic invariants and proving the relations between them. The following two Sections, \ref{Sec:SU} and \ref{Sec:SO}, treat respectively the cases of fusion algebras associated with $SU_q(2)$ and with $SO_q(3)$; in particular we compute all the associated invariants introduced in Section \ref{sec:asymptotic}. In Section \ref{sec:qdef} we establish explicit formulas for the uniform exponential  growth rate for arbitrary $q$-deformations. Finally in Section \ref{sec:UF} we study  the uniform exponential growth rates for fusion algebras associated with the free unitary quantum groups $U_F^+$.

\section{Fusion algebras --  definitions and basic facts}
\label{sec:fusiondef}

Throughout this paper we will consider fusion algebras in the sense of \cite{HiaiIzumi} or \cite[Definition 2.1]{Kyed}. Let us recall the definition.

\begin{deft}\label{def1}
A \emph{fusion algebra} $(R,d)$ is given by the following data: an index set $I$ (called the set of \emph{irreducible objects} of $R$) with a distinguished element $e$ and involution $\alpha \mapsto \bar{\alpha}$, the structure of a unital ring on $\ZZ[I]$ (denoted $R$)  and a \emph{dimension function} $d\colon R \to \RR$, such that
\begin{itemize}
	\item[(i)] $e$ is the unit of $R$;
	\item[(ii)] for any elements $\xi, \eta \in I$, there is a family  $(N_{\xi, \eta}^\alpha)_{\alpha \in I}$ of \emph{non-negative} integers (of which only finitely many are non-zero), such that 
	\[ \xi \eta = \sum_{\alpha \in I} N_{\xi, \eta}^\alpha \alpha;\]
	\item[(iii)] the involution extends to a $\ZZ$-linear antimultiplicative involutive map on $R$ (still denoted by the same symbol);
	\item[(iv)] Frobenius reciprocity holds, i.e.~for all $ \xi, \eta, \alpha \in I$ we have 
	\[ N_{\xi, \eta}^\alpha = N_{\alpha, \bar{\eta}}^\xi = N_{\bar{\xi}, \alpha}^\eta;\] 
	\item[(v)] $d$ is multiplicative, $\ZZ$-linear and $d(\overline{\alpha})=d(\alpha)\ge 1$ for $\alpha\in I$.
\end{itemize}
\end{deft}

Note in particular that we view the dimension function as a part of the data of the fusion algebra; in general the choice of $d$ (for a given $R=\mathbb{Z}[I]$ -- unital, involutive ring as above) is highly non-unique. If however there is a natural choice of dimension function clear from the context, or if the property we discuss does not depend on $d$, we will simply write $R$ instead of $(R,d)$.

Let $(R,d)$ be a fusion algebra.  For $r=\sum_{\alpha\in I} r_\alpha \alpha\in R$ we define the \emph{support} of $r$ as $\supp(r)=\{\alpha\in I\,|\, r_\alpha \neq 0\}$. For $r\in R, \alpha\in I$ we write $\alpha\subseteq r$ if $\alpha\in \supp(r)$. If $A\subseteq I$ then $A^c=I\setminus A$, $\ov{A}=\{\ov{\alpha}\mid \alpha\in A\}$, cardinality of $A$ is $\# A\in \ZZ_+\cup \{+\infty\}$ and the \emph{size} of $A$ is given by
\[
|A|=\sum_{\alpha\in A} d(\alpha)^2.
\]

We say that $R$ is \emph{finitely generated} if there exists a finite  subset $X\subseteq I$ such that $\ov{X}=X$ and for any $\alpha\in I$ there exist $x_1,\dotsc x_n\in X$ such that $\alpha\subseteq x_1\cdots x_n$. In this situation we will say that $X$ \emph{generates} $R$. Let stress that we will always assume that generating set $X$ is symmetric, i.e.~$\ov{X}=X$. Note that this property does not depend on the choice of dimension function $d$.
%\begin{definition}{\cite[Section 3]{Kyed}}
This is easily seen to be equivalent to the definition given in \cite[Section 3]{Kyed}, calling $R$ finitely generated if  there exists a finitely supported probability measure $\mu\in \ell^1(I)$ such that $\bigcup_{n\in\NN}\supp (\mu^{\star n})=I$ and $\mu(\ov{\alpha})=\mu(\alpha)$ for $\alpha\in I$. Here we use the standard fact that the fusion algebra structure induces a convolution product on probability measures on $I$.
%\end{definition}

\begin{remark}\label{rem:product}
When $(R_i,d_i)$ are fusion algebras with the sets of irreducible objects $I_i$ $(i\in \{1,2\})$, we can define their \emph{product fusion algebra} $R$. Its unital ring structure is defined by setting $R=R_1\otimes_{\ZZ} R_2$, with the set of irreducible objects\footnote{We will write $\alpha\boxtimes \beta\in I_1\times I_2$ instead of $\alpha\otimes \beta$ to avoid confusion with the tensor product of representations.} $I=I_1\times I_2$, the unit $e_1\boxtimes e_2$, the involution $\ov{\alpha\boxtimes \beta}=\ov{\alpha}\boxtimes \ov{\beta}$ and the dimension function determined by the formula $d(\alpha\boxtimes \beta)=d_1(\alpha)d_2(\beta)$ for $\alpha\in I_1,\beta\in I_2$. One easily sees that $N^{\alpha\boxtimes \alpha'}_{\xi\boxtimes \xi',\eta\boxtimes \eta'}=N^{\alpha}_{\xi,\eta} N^{\alpha'}_{\xi',\eta'}$ for $\alpha,\xi,\eta\in I_1,\alpha',\xi',\eta'\in I_2$ and if $R_1,R_2$ are finitely generated, then so is $R$.
\end{remark}

\begin{definition} \label{def:length}
	Let $R$ be a finitely generated fusion algebra, and let  $X\subseteq I$ be a finite generating set. The associated length function is defined via%\footnote{What should be the length of $e$? Do we want $e\in X$ or not? Do we want $X=\ov{X}$?}
	\[
	\ell_X\colon I\setminus\{e\}\ni \alpha\mapsto \min\{n \in \mathbb{N}\mid  \exists_{x_1,\dotsc,x_n \in X} \colon \alpha\subseteq x_1\cdots x_n \}\in \NN,
	\]
	additionally setting $\ell_X(e)=0$. %\footnote{It might be natural to consider symmetric sets $X$, but for that we need to check classical group literature.}
The balls and spheres in $I$ of radius $n \in \ZZ_+$ (with respect to $X$) are defined respectively via
	\[
	B_X(n)=\{\alpha\in I\mid  \ell_X(\alpha)\le n\},\quad
	S_X(n)=\{\alpha\in I\mid \ell_X(\alpha)=n\}.
	\]
\end{definition}

The following lemma is a direct consequence of \cite[Proposition 1.2 (6)]{HiaiIzumi}.

\begin{lemma} \label{lem:prop}
	Let $R$ be a fusion algebra, with $I$ denoting the set of irreducible objects. %We say that $R$ satisfies the finite propagation property if 
	Then for any finite set $X \subseteq I$ there exists a constant $M_X\in \NN$ such that for any $v \in I$ and $x \in X$ we have $\# (\textup{supp}\,(vx))\leq M_X$. 	
\end{lemma}

Note that the dimension function plays no role in the statement above (although it features in the proof in \cite{HiaiIzumi} and one can use $d$ to get an explicit bound on $M_X$). It is easy to see that one could also replace in the formula above the condition  $\# (\textup{supp}\,(vx))\leq M_X$ by its right equivalent  $\#(\textup{supp }(xv))\leq M_X$ (by using the `antimultiplicativity' of the adjoint operation).

We have already mentioned that the fusion algebra structure induces a convolution on measures in $\ell^1(I)$. Furthermore, to every probability measure $\mu \in \ell^1(I)$ one can associate a convolution operator $\lambda_\mu \in \B(\ell^2(I))$. We say that the fusion algebra $(R,d)$ is \emph{amenable} if for every probability measure $\mu$ as above we have $\|\lambda_\mu\|=1$ (see \cite[Theorem 4.6]{HiaiIzumi}, \cite[Theorem 3.3]{Kyed} or \cite[Section 2.7]{NeshveyevTuset}). Amenable fusion algebras admit `minimal' dimension functions, in the sense made precise below.

\begin{proposition}\label{prop:amenmin}
Let $(R,d)$ be an amenable fusion algebra and suppose that we have another dimension function $d'\colon I \to \RR_+$ such that $(R,d')$ is a fusion algebra. Then for every $i \in I$ we have $d(i) \leq d'(i)$.	
\end{proposition}	

\begin{proof}
	In \cite[Proposition 2.7.7]{NeshveyevTuset} this statement is established in the context of fusion algebras arising from $C^*$-tensor categories (see the next subsection). An inspection of the proof shows that it remains valid in a potentially greater generality we consider here.
\end{proof}

\subsection{Discrete/compact quantum groups, rigid $C^*$-tensor categories and related fusion algebras}\label{sec:discrete/compact}

The key examples of fusion algebras studied in this paper arise from representation theory of compact quantum groups. We will mostly follow the notation and terminology of \cite{NeshveyevTuset}. So let $\GG$ be a \emph{compact quantum group} in the sense of Woronowicz, studied via the associated `algebras of functions', namely the Hopf $^*$-algebra $\Pol(\GG)$ and its universal $\mrm{C}^*$-completion $\mrm{C}^u(\GG)$. By $\Irr(\GG)$ we denote the set of equivalence classes of \emph{irreducible representations} of $\GG$, and by\footnote{Quantum dimension is also often denoted by $d(\alpha)=\dim_q(\alpha)$.} $d\colon\Irr(\GG) \to \left[1, \infty\right)$ the \emph{quantum dimension} function of $\GG$. Once we choose a representative $U^{\alpha}$ of class $\alpha\in\Irr(\GG)$, we have $d(\alpha)=\Tr(\uprho_\alpha)$, where $\uprho_\alpha$ is the unique positive, invertible morphism between $U^{\alpha}$ and $(U^{\alpha})^{ cc}$ which satisfies $\Tr(\uprho_\alpha)=\Tr(\uprho_\alpha^{-1})$ (\cite[Definition 1.4.1]{NeshveyevTuset}). We will denote by $(R(\GG),d)$ the fusion algebra of (equivalence classes of) finite-dimensional representations of $\GG$, equipped with the quantum dimension function mentioned above, see \cite[Definition 2.7.2]{NeshveyevTuset}. This means that $R(\GG)= \ZZ[\Irr(\GG)]$ as an abelian group, product of $R(\GG)$ is induced by the tensor product of representations and the set of irreducible objects in $R(\GG)$ is identified with $\Irr(\GG)$. Unless said otherwise, we will always equip $R(\GG)$ with the quantum dimension function. So, if for example we want to equip ring $R(\GG)$ with the classical dimension function $\dim$, we will indicate it by writing $(R(\GG),\dim)$. We have $d = \dim$ if and only if the quantum group $\GG$ is of \emph{Kac type}.

It is easy to see that $R(\GG)$ is finitely generated if and only if $\GG$ is a compact matrix quantum group.

A special class of compact quantum groups is given by duals of classical discrete groups. Indeed, if $\Gamma$ is a discrete group, we can treat $\ZZ\Gamma$ as a fusion algebra with basis $\Gamma$ and dimension function equal to $1$ on the basis elements. In the language developed above we would have $\ZZ\Gamma=R(\wh{\Gamma})$, where $\wh{\Gamma}$ is the dual compact quantum group. Yet another special class appears if we consider classical compact groups $G$, for which the quantum dimension function equals the usual dimension of irreducible representations of $G$. The fusion algebra $R(G)$ is then always amenable, whereas $R(\wh{\Gamma})$ is amenable if and only if $\Gamma$ is amenable (see the last part of this subsection). Further specific examples will be presented in later sections.

Another source of fusion algebras is provided by rigid \cst-tensor categories (see \cite[Chapter 2]{NeshveyevTuset}). Let $\mf{C}$ be such a category. Recall that this in particular means that there is a way of tensoring objects $U\otimes V\in \mf{C}$, taking direct sums $U\oplus V\in \mf{C}$, and taking the conjugate object $\ov{U}\in \mf{C}$ $(U,V\in \mf{C})$; we can also speak of subobjects. An object $U\in \mf{C}$ is said to be \emph{simple} if $\dim \oon{End}(U)=1$, or equivalently $U$ does not admit proper subobjects.

With category $\mf{C}$ one can associate a fusion ring $R(\mf{C})$, which is the universal ring generated by equivalence classes of simple objects, with the ring structure coming from the direct sum and tensor product in $\mf{C}$ (see \cite[Definition 2.7.2]{NeshveyevTuset}). Thus, the set of irreducible objects in $R(\mf{C})$  is the set of (equivalence classes of) simple objects of $\mf{C}$. Taking the conjugate object $U\mapsto \ov{U}$ gives rise to an involution on $R(\mf{C})$. Once we equip $R(\mf{C})$ with any dimension function, we obtain a fusion algebra in the sense of Definition \ref{def1}. While the choice of dimension function is not unique, there is always a canonical choice given by  the \emph{intrinsic dimension} of $\mf{C}$ (\cite[Definition 2.2.11]{NeshveyevTuset}).

With any compact quantum group $\GG$ one can associate a rigid $\cst$-tensor category $\Rep(\GG)$ of finite dimensional unitary representations of $\GG$. Then $R(\Rep(\GG))$ is equal to the fusion algebra $R(\GG)$ described above (and the intrinsic dimension of $\Rep(\GG)$ is equal to the quantum dimension \cite[Example 2.2.13]{NeshveyevTuset}).\\

Not surprisingly, the notion of amenability for the fusion algebra $R(\GG)$ is related to  amenability of the dual discrete quantum group $\wh \GG$. More precisely, $\wh \GG$ is  amenable %\footnote{Strong amenability and amenability of $\wh\GG$ are equivalent by \cite{Tomatsu}, or an unpublished work of Blanchard-Vaes.} 
if and only if $(R(\GG),\dim)$ is an amenable fusion algebra  (\cite[Theorem 2.7.10]{NeshveyevTuset}).

Amenability of the fusion algebra $R(\GG)$ (equipped with the quantum dimension function) is also related to amenability-like properties of a number of related objects. Specifically, the following conditions are equivalent:
\begin{itemize}
\item $R(\GG)$ is amenable;
\item $\whG$ is centrally strongly amenable;
\item $\whG$ is strongly amenable and $\GG$ is of Kac type;
\item $\mrm{C}^*$-tensor category $\Rep(\GG)$ is amenable;
\item Drinfeld double $D(\whG)$ is (strongly) amenable.
\end{itemize}
(See \cite[Theorem 7.8]{DKVAveraging}, \cite[Proposition 2.7.7]{NeshveyevTuset} and \cite[Definition 7.1]{Brannan}, \cite[Definition 4.2]{DKVAveraging}, \cite[Definition 2.7.6]{NeshveyevTuset} for relevant definitions).
 
\subsection{$q$-numbers} 
 
Throughout the paper, we will use the notion of $q$-numbers. For $0<q<1$ and $x\in \mathbb{C}$, denote $[x]_q=\tfrac{q^{-x}-q^x}{q^{-1}-q}$, it will also be convenient to write $[x]_1=x$. We will need the following lemma.

\begin{lemma}\label{lemma9}
Let $m\in\NN$, $0<q<q'<1$. The following inequalities hold:
\begin{itemize}
\item[(i)] $\tfrac{[m+1]_q}{[m]_q} > \tfrac{[m+1]_{q'}}{[m]_{q'}}$;
\item[(ii)] $\tfrac{[m]_q}{[2]_q} > \tfrac{m}{2}$.
\end{itemize}
\end{lemma}

\begin{proof}
For $m\in\NN$, set $f_m\colon (0,1)\ni q\mapsto q^{-m}-q^m\in (0,\infty)$. Then 
\begin{equation}\label{eq18}
f_{m+1}'(q)-f_m'(q)=mq^{-1-m} (1-q) (-q^{-1}+q^{2m})-q^{-m-2}-  q^m <0.
\end{equation}
Let $0<q<q'<1$. Condition (i) is equivalent to $[m]_{q'} [m+1]_q > [m]_q [m+1]_{q'}$, i.e.\ to $f_m(q') f_{m+1}(q) > f_m(q) f_{m+1}(q')$ and further to
\[
\bigl( \int_{q}^{q'} f_{m}'(t)\md t +f_m(q)\bigr) f_{m+1}(q) > f_m(q)\bigl( \int_{q}^{q'} f_{m+1}'(t)\md t +f_{m+1}(q)\bigr).
\]
The latter amounts to the inequality $f_{m+1}(q) \int_{q}^{q'}f_m'(t)\md t > f_m(q) \int_{q}^{q'} f_{m+1}'(t)\md t$, which holds by \eqref{eq18} and the easily checked inequality $f_m(q)\le f_{m+1}(q)$. Passing with $q'$ to $1$ gives $\tfrac{[m+1]_q}{[m]_q} > \tfrac{m+1}{m}$, hence also
\[
\tfrac{[m]_q}{[2]_q} = 
\tfrac{[m]_q}{[m-1]_q} \tfrac{ [m-1]_q }{[m-2]_q}\cdots \tfrac{[3]_q}{[2]_1} > 
\tfrac{m}{m-1} \tfrac{ m-1 }{m-2}\cdots \tfrac{3}{2}=
\tfrac{m}{2}.
\]
\end{proof}

\section{Definitions and general properties of asymptotic invariants for fusion algebras}
\label{sec:asymptotic}

Throughout this section let us fix a fusion algebra $(R,d)$, with $I$ denoting the set of irreducible objects.

\subsection{Uniform F\o lner constant}
The first invariant associated to $(R,d)$ we will consider is the \emph{uniform (inner) F\o lner constant}. As well-known both classically and in the quantum world, the notion of F\o lner sets is closely related to amenability; we will see an instance of this below.

To that end we need to recall the notions of boundary $\partial_X (A)$ and inner boundary $\partial^{\inn}_X (A)$ for sets $A,X \subseteq I$ (see \cite{Kyed}, \cite{KyedThom}).

\begin{definition} \cite[Definition 3.2]{Kyed}	
Let $X,A$ be finite subsets of $I$. The \emph{boundary} of $A$ with respect to $X$ is the set
	\[
	\partial_X (A)=\{ \alpha\in A\,|\, \exists_{x\in X}\; \supp(\alpha x) \nsubseteq A\}\cup
	\{ \alpha\in A^c\,|\, \exists_{x\in X}\; \supp(\alpha x) \nsubseteq A^c\}
	\]
	and \emph{inner boundary} of A with respect to $X$ is
	\[
	\partial_X^{\inn}(A)=\{\alpha\in A\,|\, \exists_{x\in X}\;
	\supp(\alpha x)\nsubseteq A\}.
	\]	
\end{definition}

The non-uniform case of the following definition appears in \cite{Kyed}. The uniform one is modelled on \cite{Arzhantseva}; note however that the latter paper considers only inner F\o lner constants. 

\begin{definition}
The \emph{F\o lner constant of $(R,d)$ with respect to the generating set $X$} is defined as
	\[
	\Fol_X(R,d)=\inf_A \tfrac{ |\partial_X A|}{|A|},
	\]
	where $A$ runs over non-empty finite subsets of $I$. The \emph{uniform F\o lner constant} of $(R,d)$ is
	\[
	\Fol(R,d)=\inf_X \Fol_X(R,d)
	\]
	where $X$ runs over all finite generating subsets of $R$. % We will sometimes write $\Fol_X(R,d)$ and $\Fol(R, d)$ to emphasise the choice of dimension function $d$.
	We can analogously define the \emph{(uniform) inner F\o lner constants}  $\FolinnX(R,d)$  and $\Folinn(R,d)$, using the inner boundary in the first formula displayed above.
\end{definition}

Since $\partial^{\inn}_X(A)\subseteq \partial_X(A)$, we immediately see that
\[
\FolinnX(R,d)\le \Fol_X(R,d),\quad
\Folinn(R,d)\le \Fol(R,d).
\]

We will use both of these invariants. Their relationship is not yet fully understood even in the case of discrete groups -- see  \cite[Proposition A.1]{LPV} and discussion after that, concerning the connection between the uniform inner F\o lner constant and the so-called \emph{uniform isoperimetric constant}. However, we record here an easy observation, following Lemma \ref{lem:prop}.

\begin{proposition}\label{prop:C_X}
	For every finite, symmetric, nonempty set $X \subseteq I$ there exists $C_X>0$ such that for every finite set $A\subseteq I$ we have $|\partial_X (A)| \leq C_X |\partial^{\inn}_X(A)|$. In particular if $X$ is generating, then 
\[
\Fol_X(R,d)\le C_X \FolinnX(R,d).
\]
\end{proposition}

\begin{proof}
	Take $\alpha\in A^c \cap \partial_X(A)$. Then there is $x\in X$ and $\beta\in A$ such that $\beta\subseteq \alpha x$, equivalently $N^{\beta}_{\alpha, x}>0$. But then $N^{\alpha}_{\beta, \ov{x}}>0$, i.e.~$\alpha\subseteq \beta \ov{x}$ and consequently $d(\alpha)\le d(\beta)d(\ov{x})$. Furthermore, since $\alpha\in A^c$ and $\alpha\subseteq \beta \ov{x}$, we have $\beta\in \partial_X^{\inn}(A)$. Using the number $M_X$ introduced in Lemma \ref{lem:prop}, we have
\begin{align*}
|\partial_X(A)|=	\sum_{\xi \in\partial_X(A)}d(\xi)^2 &\le 
	\sum_{\xi \in \partial_X^{\inn}(A)}d(\xi)^2+\sum_{\beta\in \partial_X^{\inn}(A)}\sum_{x\in X}M_X d(\beta)^2 d(\ov x)^2 \\
	&=
	\bigl(1+\sum_{x\in X}M_Xd(x)^2\bigr) \sum_{\xi\in \partial_X^{\inn}(A)} d(\xi)^2=
	\bigl(1+M_X |X|\bigr) |\partial_X^{\inn}(A)|,
\end{align*}
thus it is enough to put $C_X=1+M_X |X|$.
\end{proof}

Kyed in \cite[Theorem 3.3]{Kyed} characterises amenability of $R$ via the following F\o lner-like condition.
\begin{theorem}\label{th:Kyed}
	The following conditions are equivalent:
	\begin{itemize}
		\item[(i)] $(R,d)$ is amenable;
		\item[(ii)] for every non-empty finite set $X\subseteq I$ and $\eps>0$ there exists a finite subset $F\subseteq I$ such that
		\begin{equation}\label{eq3}
		|\partial_X(F)|< \eps |F|.
		\end{equation}
	\end{itemize}
\end{theorem}

It is easy to see (e.g.~by considering $X\cup \ov{X}$), that in the second condition it suffices to consider symmetric $X\subseteq I$. Thanks to Proposition \ref{prop:C_X}, we can in fact formulate the analogous result using inner boundaries.

\begin{proposition}\label{prop2}
	Suppose that $(R,d)$ is a fusion algebra.
	The following conditions are equivalent:
	\begin{itemize}
		\item[(i)] $(R,d)$ is amenable;
		\item[(ii)] for every non-empty  finite set $ X\subseteq I$ and $\eps>0$ there exists a finite subset $F\subseteq I$ such that
		\begin{equation}\label{eq4}
	|\partial_X^{\inn}(F)|< \eps |F|.
		\end{equation}
	\end{itemize}
\end{proposition}

\begin{proof}
By Theorem \ref{th:Kyed} it is enough to prove that (ii) implies (i). Choose a non-empty finite $ X=\ov{X}\subseteq I$ and $\eps>0$. By assumption, we can find finite set $F\subseteq I$ such that
\[
|\partial_{X}^{\inn}(F)| < \tfrac{ \eps}{C_X} |F|,
\]
where $C_X>0$ is the number introduced in Proposition \ref{prop:C_X}. Since
\[
|\partial_X(F)|\le C_X |\partial_X^{\inn}(F)| <\eps |F|,
\]
Theorem \ref{th:Kyed} ends the claim.
\end{proof}

\begin{comment}
%old proof
\begin{proof}
By Theorem \ref{th:Kyed} it is enough to prove that (ii) implies (i). Choose a non-empty finite $ X=\ov{X}\subseteq I$ and $\eps>0$. Let $M_X>0$ be the `finite propagation constant' of $X$ appearing in Lemma \ref{lem:prop}. By \eqref{eq4} there exists a finite subset $F\subseteq I$ such that
	\[
	\sum_{\xi\in \partial_X^{\inn}(F)} d(\xi)^2 < \tfrac{\eps}{1+M_X\sum_{x\in X}d(x)^2} \sum_{\xi\in F} d(\xi)^2.
	\]
	Clearly
	\[
	\partial_X(F)=\partial_X^{\inn}(F)\cup\{\alpha\in F^c \,|\, 
	\exists_{x\in X}\; \supp(\alpha x)\nsubseteq F^c\}.
	\]
	Take $\alpha\in F^c$ such that $\exists_{x\in X}$ and $\exists_{\beta\in F}$ so that $\beta\subseteq \alpha x$, equivalently $N^{\beta}_{\alpha , x}>0$. But then $N^{\alpha}_{\beta , \ov{x}}>0$, i.e.\ $\alpha\subseteq \beta \ov{x}$ and consequently $d(\alpha)\le d(\beta)d(\ov{x})$. Furthermore, since $\alpha\in F^c$ and $\alpha\subseteq \beta \ov{x}$, we have $\beta\in \partial_X^{\inn}(F)$. It follows that
\begin{align*}
	\sum_{\xi \in\partial_X(F)}d(\xi)^2 &\le 
	\sum_{\xi \in \partial_X^{\inn}(F)}d(\xi)^2+\sum_{\beta\in \partial_X^{\inn}(F)}\sum_{x\in X}M_X d(\beta)^2 d(\ov x)^2 \\
	&=
	\bigl(1+\sum_{x\in X}M_Xd(x)^2\bigr) \sum_{\xi\in \partial_X^{\inn}(F)} d(\xi)^2 <\eps \sum_{\xi\in F}d(\xi)^2.
\end{align*}
Theorem \ref{th:Kyed} ends the claim.
\end{proof}
\end{comment}

If $(R,d)$ is finitely generated, we can check the second of the conditions appearing in results above only for finite generating sets. We formulate this fact as the following proposition.

\begin{proposition}\label{prop10}
Let $(R,d)$ be a finitely generated fusion algebra. The following conditions are equivalent:
\begin{itemize}
	\item[(i)] $(R,d)$ is amenable;
	\item[(ii)] for every  finite generating set $ X\subseteq I$ we have $\Fol_X(R,d)=0$;
	\item[(iii)] for every  finite generating set $ X\subseteq I$ we have $\FolinnX(R,d)=0$.
\end{itemize}
\end{proposition}
\begin{proof}
By Theorem \ref{th:Kyed} and Proposition \ref{prop2} it suffices to observe that (iii)$\Longrightarrow$(ii)$\Longrightarrow$(i).

Implication (iii)$\Longrightarrow$(ii) follows immediately from Proposition \ref{prop:C_X}.
% exactly the arguments in the proof of Proposition \ref{prop2}. 

If (ii) holds, we can follow line by line the proof of the implication (FC3)$\Longrightarrow$(FC1) of \cite[Theorem 3.3] {Kyed} to obtain a `non-degenerate version' of property (FC1) of that theorem, and conclude by \cite[Theorem 4.6]{HiaiIzumi}.
	
\end{proof}

We then obtain an immediate corollary, connecting amenability of $(R,d)$ to vanishing of the uniform F\o lner constants.

\begin{corollary}\label{prop3}
	If $(R,d)$ is an amenable finitely generated fusion algebra then $\Fol(R,d)= \Folinn (R,d)=0$.
\end{corollary}

It is not true that vanishing of the uniform F\o lner constants is equivalent to amenability. Examples can be found already for classical groups: \cite[Proposition 13.3]{Arzhantseva} shows that certain Baumslag-Solitar groups are non-amenable, but have uniform inner F\o lner constant equal $0$.

\subsection{Uniform exponential growth rate}

In this section we introduce the (uniform) exponential growth rate for finitely generated fusion algebra, closely connected to the notion studied in \cite{BanicaVergnioux} in the case of Kac type compact quantum groups (Remark \ref{rem:BV}), and in \cite{Arzhantseva} in the case of classical discrete groups (Remark \ref{rem:Aetal}). It gives an upper bound on the uniform (inner) F\o lner constant (Proposition \ref{prop1}) and is often easier to compute than the F\o lner constants. In particular, in Section \ref{sec:qdef} we will compute it for the fusion algebras associated with %$SU_q(2),\operatorname{SO}_q(3)$ and 
$q$-deformations of compact, semisimple, simply connected Lie groups, and in Section \ref{sec:UF} for the fusion algebras associated with free unitary quantum groups.

Recall that if $(R,d)$ is a finitely generated fusion algebra with finite generating set $X$, then in Definition \ref{def:length} we have introduced the notion of length $\ell_X$ and the corresponding notions of spheres and balls. Recall also that generating set $X$ is always assumed to be symmetric $\ov{X}=X$. We begin with a standard lemma.

\begin{lemma}\label{lemma2}
Let $X\subseteq I$ be a finite generating set, and suppose that $I$ is infinite. Then the limits $ \lim_{n\to\infty} \sqrt[n]{|B_X(n)|}$ and $ \lim_{n\to\infty} \sqrt[n]{|S_X(n)|}$ exist, and belong  to $ [1, \infty)$.
\end{lemma}

\begin{proof}
	Fix $n,m\in\NN$. Take $\beta\in B_X(n),\gamma\in B_X(m)$ and set $\Delta(\beta,\gamma)=\{\alpha\in B_X(n+m)\,|\, \alpha\subseteq \beta\gamma\}$. Clearly $\bigcup_{\beta\in B_X(n),\,\gamma\in B_X(m)} \Delta(\beta,\gamma)=B_X(n+m)$. Next,
	\[
	\sum_{\alpha\in \Delta(\beta,\gamma)}d(\alpha) \le d(\beta) d(\gamma)
	\]
	as $ d(\beta\gamma)=d(\beta)d(\gamma)$ and each $\alpha \in \Delta(\beta,\gamma)$ appears (at least once!) in the basis decomposition of $\beta \gamma$. It follows that
	\[
	\sum_{\alpha\in \Delta(\beta,\gamma)} d(\alpha)^2 + \sum_{\alpha,\alpha'\in \Delta(\beta,\gamma)\colon\alpha\neq \alpha'}
	\!\!\!\! d(\alpha) d(\alpha') 
	=
	\bigl(
	\sum_{\alpha\in \Delta(\beta,\gamma)}d(\alpha) 
	\bigr)
	\bigl(
	\sum_{\alpha'\in \Delta(\beta,\gamma)}d(\alpha') 
	\bigr)\le d(\beta)^2 d(\gamma)^2
	\]
	and since $d$ is a non-negative function on $I$ we have
	\[
	\sum_{\alpha\in \Delta(\beta,\gamma)} d(\alpha)^2 
	\le d(\beta)^2 d(\gamma)^2.
	\]
	Consequently
	\[\begin{split}
	\sum_{\alpha\in B_X(n+m)} d(\alpha)^2\le 
	\sum_{\beta\in B_X(n)}\sum_{\gamma\in B_X(m)} \sum_{\alpha\in \Delta(\beta,\gamma)} d(\alpha)^2
	\le \sum_{\beta\in B_X(n)}\sum_{\gamma\in B_X(m)} d(\beta)^2 d(\gamma)^2.
	\end{split}\]
	This shows $|B_X(n+m)|\le |B_X(n)|\,|B_X(m)|$ and a standard argument proves existence of the limit $\lim_{n\to\infty} \sqrt[n]{|B_X(n)|}$.  An analogous reasoning shows that $\lim_{n\to\infty}\sqrt[n]{|S_X(n)|}$ also exists.
\begin{comment}
	Fix $n,m\in\NN$. Take $\beta\in S_X(n),\gamma\in S_X(m)$ and set $\Delta(\beta,\gamma)=\{\alpha\in S_X(n+m)\,|\, \alpha\subseteq \beta\gamma\}$. Clearly $\bigcup_{\beta\in S_X(n),\,\gamma\in S_X(m)} \Delta(\beta,\gamma)=S_X(n+m)$. Next,
	\[
	\sum_{\alpha\in \Delta(\beta,\gamma)}d(\alpha) \le d(\beta) d(\gamma)
	\]
	as $ d(\beta\gamma)=d(\beta)d(\gamma)$ and each $\alpha \in \Delta(\beta,\gamma)$ appears (at least once!) in the basis decomposition of $\beta \gamma$. It follows that
	\[
	\sum_{\alpha\in \Delta(\beta,\gamma)} d(\alpha)^2 + \sum_{\alpha,\alpha'\in \Delta(\beta,\gamma)\colon\alpha\neq \alpha'}
	\!\!\!\! d(\alpha) d(\alpha') 
	=
	\bigl(
	\sum_{\alpha\in \Delta(\beta,\gamma)}d(\alpha) 
	\bigr)
	\bigl(
	\sum_{\alpha'\in \Delta(\beta,\gamma)}d(\alpha') 
	\bigr)\le d(\beta)^2 d(\gamma)^2
	\]
	and since $d$ is a non-negative function on $I$
	\[
	\sum_{\alpha\in \Delta(\beta,\gamma)} d(\alpha)^2 
	\le d(\beta)^2 d(\gamma)^2.
	\]
	Consequently
	\[\begin{split}
	\sum_{\alpha\in S_X(n+m)} d(\alpha)^2\le 
	\sum_{\beta\in S_X(n)}\sum_{\gamma\in S_X(m)} \sum_{\alpha\in \Delta(\beta,\gamma)} d(\alpha)^2
	\le \sum_{\beta\in S_X(n)}\sum_{\gamma\in S_X(m)} d(\beta)^2 d(\gamma)^2.
	\end{split}\]
	This shows $|S_X(n+m)|\le |S_X(n)|\,|S_X(m)|$ and a standard argument proves existence of the limit $\lim_{n\to\infty} \sqrt[n]{|S_X(n)|}$.
	\end{comment}
\end{proof}

The limits above exist also when $I$ is finite; we then have $ \lim_{n\to\infty} \sqrt[n]{|B_X(n)|}=1$, but $ \lim_{n\to\infty} \sqrt[n]{|S_X(n)|}=0$.

We can now introduce the definition of the (uniform) exponential growth rate.

\begin{definition}
	Let $(R,d)$ be a fusion algebra generated by a finite set $X\subseteq I$. The \emph{exponential growth rate} of $R$ with respect to $X$ is given by
	\[
	\omega_X(R,d)=\lim_{n\to\infty} \sqrt[n]{|B_X(n)|},
	\]
	and the \emph{uniform exponential growth rate} by
\[
\omega(R,d)=\inf_X \omega_X(R,d),
\]
	where $X$ runs over all finite generating sets for $R$. If $\omega(R,d)>1$, then we say that $R$ has \emph{uniform exponential growth}.
\end{definition}

\begin{remark}\label{rem:BV}
		Let $\bbGamma$ be a discrete quantum group such that $\wh{\bbGamma}$ admits a fundamental representation. %(e.g.\ $R(\bbGamma)$ is finitely generated).
		 Let $X\subseteq \Irr(\wh{\bbGamma })$ be a finite generating set. Our definition of $\omega_X(R(\wh \bbGamma),d)$, for $\bbGamma$ unimodular (so that $\wh  \bbGamma$ is of Kac type), agrees with the notion of ratio of exponential growth from \cite[Section 4]{BanicaVergnioux}. Indeed, we have (using notation from \cite{BanicaVergnioux}) $|B_X(n)|=b_n$. Then the continuity of $\log\colon \RR_{\ge 1}\rightarrow \RR$ yields immediately
		\[
		\log(\omega_X(R(\wh \bbGamma),d))=\lim_{n\to\infty} \tfrac{\log(|B_X(n)|)}{n}=\lim_{n\to\infty} \tfrac{\log(b_n)}{n}.
		\]
		In particular, Lemma \ref{lemma2} shows that the above limit always exists.
\end{remark}

\begin{remark}\label{rem:Aetal}
	 Let $\Gamma$ be a discrete group with a finite generating set $X\subseteq \Gamma$. Then our definition of $\omega_X(\ZZ\Gamma)$ agrees with the notion of the exponential growth rate of $\Gamma$ with respect to $X$ studied in \cite{Arzhantseva}, and similarly respective notions of uniform exponential growth rate coincide.
\end{remark}

In the next result we show that one can calculate $\omega_X(R,d)$ using the size of spheres $|S_X(n)|\,(n\in\NN)$. Next we use it to compute the growth rate for the product fusion algebra (Remark \ref{rem:product}). 

\begin{lemma}\label{lemma3}
Let $(R,d)$ be a fusion algebra with a finite generating set $X\subseteq I$. Then $\omega_X(R,d)=\max\bigl( 1 , \lim_{n\to\infty} \sqrt[n]{|S_X(n)|}\,\bigr)$.
\end{lemma}

\begin{proof}
If $\omega_X(R,d)=1$, then $\sqrt[n]{|S_X(n)|}\le \sqrt[n]{|B_X(n)|}\xrightarrow[n\to\infty]{}1$ and the claim holds.

Consider the second case, $\omega_X(R,d)>1$. Then $\sup_{n\in\NN}|S_X(n)|=+\infty$ (as otherwise $n\mapsto |B_X(n)|$ is bounded by a polynomial function) and we can find a strictly increasing sequence $(n_k)_{k\in\NN}$ of natural numbers such that $|S_{X}(n_k)| = \max_{0\le m \le n_k } |S_{X}(m)|$. For $k\in\NN$ we have
\[
|B_X(n_k)|=\sum_{m=0}^{n_k} |S_X(m)| \le 
(n_k+1) |S_X(n_k)|.
\]
Hence using Lemma \ref{lemma2} we obtain
\begin{align*}
\omega_X(R)&=\lim_{k\to\infty}|B_X(n_k)|^{1/n_k}\le 
\lim_{k\to\infty}\bigl((n_k+1)|S_X(n_k)|\bigr)^{1/n_k}=\lim_{n\to\infty}|S_X(n)|^{1/n} \\
&\le 
\lim_{n\to\infty}|B_X(n)|^{1/n}=\omega_X(R).
\end{align*}
\end{proof}

The next proposition will be used in Section \ref{sec:qdef} to reduce computation of $\omega(R(G_q))$ to the simple case. 

\begin{proposition}\label{prop4}
Let $(R_1,d_1), (R_2,d_2)$ be finitely generated fusion algebras and $(R,d)$ the product fusion algebra. The uniform growth rate of $(R,d)$ is given by
\begin{equation}\label{eq2}
\omega(R,d)=\max\bigl( \omega(R_1,d_1), \omega(R_2,d_2)\bigr).
\end{equation}
\end{proposition}

\begin{proof}
Recall (Remark \ref{rem:product}) that $R=R_1\otimes_{\ZZ} R_2$ and the set of irreducible objects in $R$ is given by $I=I_1\times I_2=\{\alpha\boxtimes \beta\mid \alpha\in I_1,\beta\in I_2\}$. Let $X_1\subseteq I_1$ (resp.~$X_2\subseteq I_2$) be a finite generating set for $R_1$ (resp.~ for $R_2$). Then $R$ is generated by (symmetric set) $X=\{\alpha\boxtimes e,e\boxtimes \beta\mid \alpha\in X_1,\beta\in X_2\}$.
The length of $\alpha\boxtimes\beta\in I$ is given by $\ell_X(\alpha\boxtimes\beta)=\ell_{X_1}(\alpha)+\ell_{X_2}(\beta)$. Thus for $n\in\ZZ_+$ we have
\[\begin{split}
S_X(n)&=
\bigcup_{k=0}^{n} \{\gamma\boxtimes\delta \;\big|\;
\gamma\in S_{X_1}(k), \,
\delta\in S_{X_2}(n-k)\bigr\}.
\end{split}\]
Let $f_X$ be the function defined by power series $f_X(z)=\sum_{n=0}^{\infty} |S_X(n)| z^n$ (on the subset of $\CC$ where the series converges), similarly define $f_{X_1},f_{X_2}$. Then (see \cite[Theorem 3.1]{BanicaVergnioux})
\[
f_X(z) = f_{X_1} (z) f_{X_2}(z).
\]
Comparing the radii of convergence \cite[Section 10.5]{RudinRC} gives
\[
\bigl(\limsup_{n\to\infty} |S_X(n)|^{1/n}\bigr)^{-1}=
\min \bigl(
\bigl(\limsup_{n\to\infty} |S_{X_1}(n)|^{1/n}\bigr)^{-1},
\bigl(\limsup_{n\to\infty} |S_{X_2}(n)|^{1/n}\bigr)^{-1}\bigr)
\]
or taking into account Lemma \ref{lemma2}
\[
\lim_{n\to\infty} |S_X(n)|^{1/n}=
\max \bigl(\lim_{n\to\infty} |S_{X_1}(n)|^{1/n},
\lim_{n\to\infty} |S_{X_2}(n)|^{1/n}\bigr).
\]
Consequently by Lemma \ref{lemma3}
\[
\omega_X(R,d)=
\max\bigl( 
\omega_{X_1}(R_1,d_1),
\omega_{X_2}(R_2,d_2)\bigr).
\]
As $X_1,X_2$ were arbitrary finite generating sets, we have  $
\omega(R,d)\le\max\bigl( \omega(R_1,d_1), \omega(R_2,d_2)\bigr)$. Note that the argument above follows essentially  the reasoning in the proof of \cite[Theorem 3.1]{BanicaVergnioux}.

For the converse inequality, let $X=\{\alpha_i^{(1)}\boxtimes\alpha^{(2)}_i\mid 1\le i \le N\}$ be a finite generating set for $R$. For $k\in \{1,2\}$ set $X_k=\{\alpha_i^{(k)}\mid 1\le i \le N\}$. Then $X_k$ is a finite, symmetric generating set for $R_k$. Take $k=1$ and an arbitrary $n\in\ZZ_+$. If $\alpha\in B_{X_1}(n)$, then there is $\beta\in I_2$ such that $\alpha\boxtimes \beta\in B_X(n)$. Consequently $|B_{X_1}(n)|\le |B_X(n)|$ and
\[
\omega(R_1,d_1)\le \omega_{X_1}(R_1,d_1)\le \omega_X(R,d).
\]
Since $X$ was an arbitrary generating set for $R$, we obtain $\omega(R_1,d_1)\le \omega(R,d)$. Similarly we check $\omega(R_2,d_2)\le \omega(R,d)$, which ends the proof.
\end{proof}

Fusion algebras of non-Kac type compact matrix quantum groups often have uniform exponential growth. We will quantify this observation in the next two results.

Let us introduce convenient terminology: if $\GG$ is a compact quantum group, then for $\alpha\in \Irr(\GG)$ we denote $\Gamma(\alpha)=\max \oon{Sp}(\uprho_\alpha)=\|\uprho_\alpha\|$ (\cite{KrajczokSoltanBounded}, \cite[Section 3.1]{KrajczokSoltan}) and $T^\sigma(\GG)=\{t\in \RR\mid \sigma^h_t=\id\}$, where $(\sigma^h_t)_{t\in\RR}$ is the \emph{modular group of the Haar integral} (\cite[Definition 2.1]{KrajczokSoltan}).

\begin{proposition}\label{prop9}
Let $\GG$ be a compact matrix quantum group which is not of Kac type and let
\[
C=\inf\{ \Gamma(\alpha)\mid \alpha\in \Irr(\GG)\colon \Gamma(\alpha)>1\} \ge 1.
\]
Then $\omega(R(\GG))\ge C^2$.
\end{proposition}

\begin{proof}
Observe that since $\GG$ is assumed to have a fundamental representation, $R(\GG)$ is finitely generated and $\omega(R(\GG))$ is well defined. If $C=1$ then the claim is trivial, hence assume $C>1$.

Let $X\subseteq \Irr(\GG)$ be a finite generating set. Then there is $\alpha\in X$ such that $\Gamma(\alpha)>1$ (as otherwise $\GG$ would have been of Kac type), hence $\Gamma(\alpha)\ge C$. By \cite[Proposition 6.1]{KrajczokSoltanBounded} for each $n\in\NN$ there is an irreducible representation $\beta_n\subseteq \alpha^{\otimes n}$ such that $\Gamma(\beta_n)=\Gamma(\alpha)^n\ge C^n$. Consequently $\beta_n\in B_X(n)$ and the claim follows from
\[
|B_X(n)|^{1/n} \ge |\{\beta_n\}|^{1/n}=
d(\beta_n)^{2/n}\ge 
\Gamma(\beta_n)^{2/n} \ge C^{2}.
\]
\end{proof}

\begin{corollary}\label{cor2}
Let $\GG$ be a compact matrix quantum group. If $T^\sigma(\GG)=\lambda \ZZ$ for $\lambda >0$, then $R(\GG)$ has uniform exponential growth with $\omega(R(\GG))\ge e^{2\pi / \lambda}$.
\end{corollary}

\begin{proof}
Since $T^\sigma(\GG)\neq \RR$, $\GG$ is not of Kac type. Take $\alpha\in \Irr(\GG)$ with $\Gamma(\alpha)>1$ and choose an orthonormal basis in $\msf{H}_\alpha$ in which $\uprho_\alpha$ is diagonal and $(\uprho_{\alpha})_{1,1}=\Gamma(\alpha)$. Then $\sigma^h_t(U^{\alpha}_{1,1})=\Gamma(\alpha)^{2it} U^{\alpha}_{1,1}\,(t\in\RR)$ (\cite[Section 1.7]{NeshveyevTuset}). Hence $\Gamma(\alpha)^{2i\lambda}=1$ and $\lambda\log(\Gamma(\alpha)) \in \pi \ZZ$. It follows that $\Gamma(\alpha) \in \{e^{ \pi k /\lambda}\mid k\in\ZZ\}$ and $\Gamma(\alpha) \ge e^{\pi/\lambda}$. Now the claim is a consequence of Proposition \ref{prop9}.
\end{proof}

Let us compare these lower bounds to the actual value of $\omega(R(\GG))$ in several cases that we will study later on in detail.

\begin{remark}\label{rem3}
For a given compact quantum group $\GG$ let us write $C_{\GG}=\inf\{\Gamma(\alpha)\mid \alpha\in\Irr(\GG)\colon \Gamma(\alpha)>1\}$ for the constant appearing in Proposition \ref{prop9}.
\begin{itemize}
\item[(i)] For $SU_q(2)$ with $q\in (0,1)$ the lower bounds on $\omega(R(SU_q(2)))$ given in Proposition \ref{prop9} and Corollary \ref{cor2} are in fact equalities (see Proposition \ref{prop14}).
\item[(ii)] More generally, consider $q \in (0,1)$ and the $q$-deformation $G_q$ for a simply connected, semisimple, compact Lie group $G$, with decomposition into simple factors $G_q\simeq\prod_{a=1}^{l} (G_a)_q$ (Proposition \ref{prop8}). Then $\omega(R(G_q))=\max_{1\le a \le l} \omega(R((G_a)_q))$ (Proposition \ref{prop4}), $C_{G_q}=\min_{1\le a \le l} C_{(G_a)_q}$. Consequently, to compare $\omega(R(G_q))$ and $C_{G_q}$ it is enough to consider the simple case. For simple $G$ not of type $B_N(N\ge 4),D_N(N\ge 5)$, we see that in fact $\omega(R(G_q))={C_{G_q}}^2$, i.e.~Proposition \ref{prop9} gives a sharp bound. Indeed, the proof of Theorem \ref{thm1} shows that in these cases the lower bound of Corollary \ref{cor1} is attained. The claim follows as $\Gamma(\lambda)=q^{-\la \lambda |2\rho\ra }$ (Proposition \ref{prop5}). If $G$ is simple and of type $B_N (N\ge 4)$ or $D_N (N\ge 5)$, then $\omega(R(G_q))> {C_{G_q}}^2$, as follows from the proof of Theorem \ref{thm1}.

For general $G$ we have $T^\sigma(G_q)=\tfrac{\pi}{\log(q)}\ZZ$ (\cite[Proposition 2.3, Theorem 4.12]{KrajczokSoltan}). Thus by Theorem \ref{thm1} the inequality of Corollary \ref{cor2} is strict unless $G_a=SU(2)$ for all $1\le a \le l$.
\item[(iii)] For $SO_q(3)$ with $q \in (0,1)$ we have $C_{SO_q(3)}=q^{-2}$ and $\omega(R(SO_q(3) ))=q^{-4}$ by Proposition \ref{prop15}, hence Proposition \ref{prop9} gives a sharp bound. On the other hand, it is easy to check that $T^\sigma(SO_q(3))=\tfrac{\pi}{\log(q)}\ZZ$, and we see that Corollary \ref{cor2} is sub-optimal. 
\item[(iv)] Consider a non-Kac type free unitary quantum group $U_F^+$ with $F\in GL_N(\CC)$ satisfying $\Tr(F^*F)=\Tr((F^*F)^{-1})$ (see Section \ref{sec:UF}) and let $0<q<1$ be defined by $q+q^{-1}=\Tr(F^*F)$. We have $C_{U_F^+}=\min( \Gamma(\alpha),\Gamma(\ov\alpha) )$ \cite[Lemma 3.8]{KrajczokSoltan} and it is easy to check that $q \le {C_{U_F^+}}^{-1}$. Using Theorem \ref{thm2} and Proposition \ref{prop16} we see that $\omega(R(U_F^+)) > C_{U_F^+}^2$, i.e.~the bound of Proposition \ref{prop9} is strict. Consequently, the same is true for Corollary \ref{cor2}.
\end{itemize}
\end{remark}

At the end of this subsection we relate the exponential growth rate to the (inner) F\o lner constant, obtaining an analog of \cite[Proposition 1.4]{Arzhantseva}.

\begin{proposition}\label{prop1}
	Let $(R,d)$ be a fusion algebra with a finite generating set $X\subseteq I$. We have $\FolinnX(R,d)\le 1-\tfrac{1}{\omega_X(R,d)}$ and consequently $\Folinn(R,d)\le 1-\tfrac{1}{\omega(R,d)}$.
\end{proposition}

\begin{proof}
Fix $n\ge 2$. We claim that
\begin{equation}\label{eq5}
\partial_X^{\inn}(B_X(n))\subseteq S_X(n).
\end{equation}
Take $\alpha\in \partial_X^{\inn}(B_X(n))$. Then $\alpha\in B_X(n)$. Suppose that $\ell_X(\alpha)<n$. Since $\alpha\in \partial^{\inn}_X(B_X(n))$, there exists $x\in X$ such that $\supp(\alpha x)\nsubseteq B_X(n)$. Thus, there is $\beta\notin B_X(n)$ so that $\beta\subseteq \alpha x$. But this forces $\ell_X(\beta) \le \ell_X(\alpha)+1\le n$. This is a contradition, which proves \eqref{eq5}. It follows that 
	\[
	\frac{ |\partial^{\inn}_X(B_X(n))|}{|B_X(n)|}\le \frac{|S_X(n)|}{|B_X(n)|}= 
	\frac{|B_X(n)|-|B_X(n-1)|}{|B_X(n)|}=1- \frac{|B_X(n-1)|}{|B_X(n)|}.
	\]
	Since $\omega_X(R,d)=\lim_{n\to\infty} \sqrt[n]{|B_X(n)|}$, we also have $\liminf_{n\to\infty} \tfrac{|B_X(n)|}{|B_X(n-1)|}\le \omega_X(R,d)$ (\cite[Theorem 3.37]{RudinPMA}). Consequently
\[\begin{split}
\FolinnX(R,d) &=\inf_{A\subseteq I, A \textup{ finite}} \frac{|\partial_X^{\inn} (A)|}{|A|}\le 
\liminf_{n\to\infty}
\frac{ |\partial^{\inn}_X(B_X(n))|}{|B_X(n)|}
\le \liminf_{n\to\infty} \bigl(1- \frac{|B_X(n-1)|}{|B_X(n)|}\bigr)\\
&=
		1-\frac{1}{\liminf_{n\to\infty} \tfrac{|B_X(n)|}{|B_X(n-1)|}}\le 
		1-\frac{1}{\omega_X(R,d)}.
	\end{split}\]
\end{proof}

\begin{remark}
If $(R,d)$ is a fusion algebra with a finite generating set $X\subseteq I$, then using a similar argument we can bound $\Fol_X(R,d)$ as follows:
\begin{equation}\label{eq9}
\Fol_X(R,d)\le \liminf_{n\to\infty} \frac{|B_X(n+1)|}{|B_X(n)|} - \frac{1}{\limsup_{n\to\infty} \frac{|B_X(n+1)|}{|B_X(n)|}}.
\end{equation}
It is not clear if the limit $\lim_{n\to\infty} \tfrac{|B_X(n+1)|}{|B_X(n)|}$ exists in general, hence we cannot express the right hand side of \eqref{eq9} using $\omega_X(R,d)$.
\end{remark}

\begin{comment}
\begin{proof}
Again fix $n \ge 2$. Let us first argue that
\begin{equation}\label{eq6}
\partial_X(B_X(n))\subseteq S_X(n)\cup S_X(n+1)
\end{equation}
Assume that $\alpha\in \partial_X(B_X(n))\setminus B_X(n)$. Then there is $x\in X$ and $\beta\in B_X(n)$ such that $\beta\subseteq \alpha x$. This means that $N^\beta_{\alpha,x}>0$ and Frobenius reciprocity gives $N^\alpha_{\beta,\ov{x}}>0$ hence $\alpha\subseteq \beta \ov{x}$. Since $\ov{X}=X$, we arrive at $n<\ell_X(\alpha)\le \ell_X(\beta)+1 \le n+1$. This, together with \eqref{eq5}, shows \eqref{eq6}. 

Take an increasing sequence $(n_k)_{k\in\NN}$ such that $\lim_{k\to\infty}\tfrac{|B_X(n_k)|}{|B_X(n_k -1)|}=\liminf_{n\to\infty} \tfrac{|B_X(n)|}{|B_X(n-1)|}$. We get
\[\begin{split}
\Fol_X(R,d) &=\inf_{A\subseteq I, A \textup{ finite}} \frac{|\partial_X (A)|}{|A|}\le 
\limsup_{n\to\infty}
\frac{ |\partial_X(B_X(n_k))|}{|B_X(n_k)|}\le \limsup_{k\to\infty} 
\frac{ |S_X(n_k)|+|S_X(n_k+1)|}{|B_X(n_k)|}
\\
 &=
 \limsup_{k\to\infty} 
\Bigl(  \frac{|B_X(n_k+1)|}{|B_X(n_k)|}
- \frac{ |B_X(n_k-1)|}{|B_X(n_k)|}\Bigr)\\
&\le 
 \limsup_{k\to\infty}\frac{|B_X(n_k+1)|}{|B_X(n_k)|}
- \frac{ 1}{  \limsup_{k\to\infty}\frac{|B_X(n_k)|}{|B_X(n_k-1)|}}\\
&\le 
 \liminf_{n\to\infty}\frac{|B_X(n+1)|}{|B_X(n)|}
- \frac{ 1}{  \limsup_{n\to\infty}\frac{|B_X(n)|}{|B_X(n-1)|}}.
\end{split}\]
\end{proof}
\end{comment}

\begin{corollary}
Let $(R,d)$ be a finitely generated fusion algebra. If $(R,d)$ is uniformly non-amenable, i.e.\ $\Folinn(R,d)>0$, then it has uniform exponential growth.
\end{corollary}

\subsection{Uniform Kazhdan constants for representations}

The article \cite{Arzhantseva} connects the F\o lner constant to  so-called \emph{Kazhdan  constants for representations}. The variations of the latter in the case of quantum groups, and the context of (central) Property (T), have been considered for example in \cite{Fima} and, in a somewhat more general version, in \cite{DSV}. 

Our framework is rather connected  to \emph{central Kazhdan constants}, related to \emph{central Property (T)} of Arano, as considered for example in \cite{VaesValvekens} and in \cite{SergeyMakoto}. It is in particular the language of the last of these papers which is directly relevant here.

Recall that $(R,d)$ is a fusion algebra  and let $\mathcal{C}_R$ denote the complexification of $R$, which we can view as a complex unital $*$-algebra with vector space structure equal to $\mathbb{C}[I]$.

\begin{definition}
Suppose that $H_\pi$ is a Hilbert space and $\pi\colon\mathcal{C}_R\to \B(H_\pi)$ is a unital $*$-representation, which contains no non-trivial \emph{invariant vectors}; this means that 
\[ \textup{Fix} (\pi) = \{\xi \in H_\pi\mid  \forall_{\alpha \in I}\, \pi(\alpha) \xi = d(\alpha) \xi\}=0.\] 
For a finite  generating set $X\subseteq I$  set
\[\Kaz(X,\pi,(R,d))= \inf_{ \xi \in H_\pi, \|\xi\|=1 } \;\;\max_{\alpha \in X} \;\;\; \tfrac{\| \pi(\alpha )\xi - d(\alpha) \xi\|}{d(\alpha)}, \]
%\[\Kaz(X,\pi,(R,d))= \inf_{ \xi \in H_\pi, \|\xi\|=1 } \;\;\max_{\alpha \in X} \;\;\; \| \pi(\alpha )\xi - d(\alpha) \xi\|d(\alpha)^{-1}, \]
and further 
\[\Kaz(\pi,(R,d)) = \inf_{X\subseteq I, X - \textup{finite, generating}} \Kaz(X,\pi,(R,d)). \]
We call $ \Kaz(\pi,(R,d))$ the \emph{Kazhdan constant of $\pi$}. 
\end{definition}

The reason for the normalisation chosen in the above definition will become apparent later on.

Consider now the left regular representation of $\cC_R$ on $\ell^2(I)$, given by a bounded linear extension of the formula ($\alpha, \beta \in I$)
\[ \lambda(\alpha) (\delta_\beta) = 
\sum_{\eta \in I}   N_{\alpha, \beta}^\eta \delta_\eta.\]
The fact that the operators $\lambda(\alpha)$ are bounded is essentially due to \cite{HiaiIzumi}. It will be more convenient for us to work with the unitarily equivalent \emph{right-regular} representation $\rho\colon \cC_R \to \B(\ell^2(I))$, given by the formula $\rho(c) = U^* \lambda(c) U$, $c \in \cC_R$, where $U\colon \ell^2(I) \to \ell^2(I)$ is the unitary mapping $\delta_\alpha \mapsto \delta_{\overline{\alpha}}\,(\alpha\in I)$. It is then easy to check via the Frobenius reciprocity, that we have ($\alpha, \beta \in I$)
\[ \rho(\alpha) (\delta_\beta) = 
\sum_{\eta \in I}   N_{\beta, \ov\alpha}^\eta \delta_\eta.\]
The following is inspired by \cite[Proposition 2.4]{Arzhantseva}. 

\begin{proposition}\label{prop:KazhFol}
	Let $(R,d)$ be a finitely generated fusion algebra and let $\rho\colon \cC_R\to \B(\ell^2(I))$ be the right regular representation of $\cC_R$. Then 
	\[ \Fol(R,d) \geq {\Kaz(\rho,(R,d) )^2}.\]
\end{proposition}
\begin{proof}
	Fix $\epsilon>0$.
	Choose a finite generating set $X\subseteq I$ and a non-empty finite subset $A\subseteq I$ such that 
	$\Fol(R,d)+\epsilon> M:= \frac{|\partial_X(A)|}{|A|}$. Set $\xi =\sum_{\beta \in A} d(\beta) \delta_\beta$, fix $\alpha \in X$ and consider the function $\phi\colon I \to \mathbb{C}$ given by  
	\[\phi(\eta) = (\rho(\alpha) \xi - d(\alpha)\xi)(\eta) = 
	\sum_{\beta \in A}  d(\beta) N_{\eta, \alpha}^\beta  - d(\alpha) d(\eta)\chi_A(\eta), \;\; \eta \in I.\] 
	Consider several possible cases. First if $\eta \in A^c \setminus \partial_X(A)$, we have $N_{\eta,\alpha}^\beta =0$ for every $\beta \in A$ (as $\eta \alpha$ does not intersect $A$), so $\phi(\eta) =0$. Second, if $\eta \in A \setminus \partial_X (A)$, we have $\eta \alpha \subseteq A$, so 
	$\sum_{\beta \in A} N_{\eta, \alpha}^\beta d(\beta)= d(\eta) d(\alpha)$ and again $\phi(\eta) =0$.
	
	For $\eta \in \partial_X (A)$ we can simply estimate
	\[ |\phi(\eta)| = \bigl|\sum_{\beta \in A}  d(\beta) N_{\eta, \alpha}^\beta  - d(\alpha) d(\eta)\chi_A(\eta)\bigr| 
	\leq d(\alpha) d(\eta),\]
	using the inequality $\sum_{\beta \in A} N_{\eta, \alpha}^\beta d(\beta)\leq d(\eta) d(\alpha)$.
	
	Note that $\|\xi\|^2 = \sum_{\beta \in A} d(\beta)^2 = |A|$ and set $\tilde{\xi} = \xi \|\xi\|^{-1}$. We then have 
	\begin{align*} \|\rho(\alpha) \tilde\xi - d(\alpha)\tilde\xi\|^2 &= |A|^{-1} \sum_{\eta \in I} |\phi(\eta)|^2 
		= |A|^{-1} \sum_{\eta \in \partial_X (A)} |\phi(\eta)|^2  \leq |A|^{-1} \sum_{\eta \in \partial_X (A)} d(\alpha)^2 d(\eta)^2 \\ &= d(\alpha)^2 |A|^{-1} |\partial_X (A)| = d(\alpha)^2 M. \end{align*}
	This implies that $\Kaz(X,\rho,(R,d))^2 \leq M$ and as $\epsilon$ was arbitrary, the proof is complete. 
\end{proof}

Note that the statement in  \cite[Proposition 2.4]{Arzhantseva} involves the `inner' F\o lner constant and apparently compensates by adding a numerical factor (claiming for $R =\mathbb{Z}\Gamma$ what in our language would be $\Folinn (R,d) \geq \frac{{\Kaz(\rho,(R,d) )^2}}2$); it is however not clear how this can be achieved (see also \cite[Appendix A]{LPV}).

We can finally connect the Kazhdan constants above to categorical Property (T), as considered for example in \cite{SergeyMakoto}. So suppose that $R$ is a fusion ring associated to a rigid $C^*$-tensor category $\mathfrak{C}$. One can reformulate \cite[Proposition 4.22]{SergeyMakoto} by saying that $\mathfrak{C}$ has Property (T) if and only if  there exist $\epsilon>0$ and a finite set $X\subseteq I$ such that for any \emph{admissible} (see \cite[Section 3]{SergeyMakoto} or \cite[Section 4]{PopaVaes}) representation $\pi\colon\mathcal{C}_R \to \B(H)$ we have $\Kaz(X,\pi,(R,d))>\epsilon$. If $R$ is finitely generated, we can naturally assume that $X$ is generating (possibly adding some elements to it). In other words, $\mathfrak{C}$ has Property (T) if the class of all admissible representations of $\mathcal{C}_R$ which do not admit invariant vectors is isolated from the trivial representation. We could thus propose the following definition, modelled after \cite[Definition 8.1]{Shalom} (see also \cite[Definition 3]{Osin}). 

\begin{definition}
	Let $\mathfrak{C}$ be a rigid $C^*$-tensor category with the associated fusion algebra $(R,d)$ and let $\mathcal{F}$ be a class of admissible representations of $\mathcal{C}_R$ which do not admit invariant vectors. We say that $\mathcal{F}$ is uniformly isolated from  the trivial representation, if there exists $\epsilon>0$ such that for every $\pi \in \mathcal{F}$ we have $\Kaz(\pi, (R,d)) >\epsilon$.
\end{definition}

Now as the representation $\rho$ considered above is admissible (see for example \cite[Corollary 4.4]{PopaVaes}), and contains no invariant vectors if and only if $I$ is infinite, we obtain the following corollary.

%{\color{red} The next Corollary is probaly false -- needs rephrasing!}

\begin{corollary}
	Suppose that $(R,d)$ is the finitely generated fusion algebra associated to a rigid $C^*$-tensor category $\mathfrak{C}$. Then if $\Fol(R,d)=0$, and $I$ is infinite, then the left regular representation of $\mathfrak{C}$ cannot be uniformly isolated from the trivial one.
\end{corollary}

%{\color{red} Again: the below needs checking!, but I think it should be OK}

Note that \cite[Theorem 4.1]{HiaiIzumi} and the proof of \cite[Proposition 5.3]{PopaVaes} imply that a rigid $\mrm{C}^*$-tensor category $\mathfrak{C}$ (with the associated fusion algebra $(R,d)$) is amenable if and only if the regular representation contains almost invariant vectors. Assuming that $(R,d)$ is finitely generated, this is the same as saying that for every finite generating set $X\subseteq I$ we have $\Kaz(X,\rho,(R,d))=0$.

%{\color{red} There is a problem here. In fact the definitions in \cite{VaesValvekens}, \cite{PopaVaes} and \cite{SergeyMakoto} seem only to  imply that if $\mathfrak{C}$ has Property (T), then there exists a finite set $F\subset \subset I$ and $\epsilon>0$ such that $\Kaz(F,\pi,R)>\epsilon$ for any admissible representation $\pi$.
	
	%I will think about ways of fixing this.}

% Connections with amenability/Property (T). 

\section{Fusion algebra of $SU_q(2)$}
\label{Sec:SU} 

In this section we will compute the invariants of the last section for the fusion algebra related to representation theory of $SU_q(2)$, with the deformation parameter $q \in (0,1]$.

We will index irreducible representations of $SU_q(2)$ by $ \ZZ_+$, so that in particular $U_1$ denotes the standard fundamental representation of $SU_q(2)$. 
The resulting fusion algebra $R$ is determined by the fusion rules of $SU(2)$. Thus $U_1$ generates $R$, the basis $I$ can be identified with non-negative integers (with $0:=e$, $1:=U_1$) and the fusion rules are simply
\begin{equation}\label{eq16}
m \otimes n = |m-n| \oplus (|m-n|+2) \oplus \cdots \oplus (m+n), \;\;\; m,n \in \ZZ_+.
\end{equation}
An arbitrary dimension function $\tilde{d}\colon \ZZ_+\to [1, \infty)$ compatible with the above fusion rules is  determined by $\tilde{d}(1)$. As noted in Subsection \ref{sec:discrete/compact}, the fusion algebra $R(SU(2))$ is amenable, so that by Proposition \ref{prop:amenmin} we must have $\tilde{d}(1)\geq 2$. For the quantum dimension of $SU_q(2)$ we have specifically $d(1)=q+ q^{-1}$. Note that we can realise all possible dimension functions on the ring above using some $q \in (0,1]$.

Let us start with a well-known lemma.  Recall the notion of $q$-number $[n]_q\,(n\in\ZZ_+)$, introduced in Section \ref{sec:discrete/compact}.

\begin{lemma} \label{dimensions}
	Let $q \in (0,1]$ and let $d\colon \ZZ_+\to \RR_+$ denote the quantum dimension function of $R(SU_q(2))$. Then $d(n)=[n+1]_q$ for all $n\in\ZZ_+$.
\end{lemma}

\begin{proof}
	The above equalities are a consequence of $d(0)=1$, $d(1) = q+q^{-1}$ and the recurrence relation
	\begin{equation}\label{eq10}
	d (n) d(1) = d(n-1) + d(n+1), \;\;\; n \in \NN, 
	\end{equation}
	which is obviously satisfied by the $q$-integers (see \cite[Formula(7)]{Kachurik}). 
	
	If $d(1)>2$, we can also offer an alternative argument, which will be useful later: set $d=d(1)$ and let
	\[ d_{\pm} = \frac{d \pm \sqrt{d^2-4}}{2}. \label{alphabeta}\]
	Then 
	\[ d(n) = \alpha d_+^n + \beta d_{-}^n, \;\;\; n \in \ZZ_+,\]
	where $\alpha, \beta \in \mathbb{R}$ are determined by the initial conditions
	\[ \alpha+ \beta =1, \;\;\alpha d_+ + \beta d_- = d.\]
	% \;\; \alpha d_+^2 + \beta d_-^2 = d^2 - 1.  \]
	Indeed, note that $d(n)$ is determined by the recurrence relation \eqref{eq10}. The standard trick with writing the recurrence relation in a matrix form shows that for each $n \in \NN$ we have 
	\[ \begin{pmatrix}  d(n+1)\\ d(n) \end{pmatrix} =  \begin{pmatrix}  d & -1\\ 1 & 0 \end{pmatrix} \begin{pmatrix}  d(n)\\ d(n-1) \end{pmatrix},\]
	and it is easy to check that the matrix in the middle has eigenvalues $\{d_+, d_-\}$.  Note that in fact we have $d_+ = q^{-1}$, $d_- =q$; thus 
	$\alpha = \frac{q^{-1}}{q^{-1} -q} $, $\beta = - \frac{q}{q^{-1} -q}$ and we recover the formula in the statement of the lemma.

	%In particular one can show that $\alpha>0$, $\beta <0$.

\end{proof}

The next lemma is also easy.

\begin{lemma}\label{lemma7}
	Let $q \in (0,1)$ and let $d\colon  \ZZ_+\to \RR_+$ denote the quantum dimension function of $R(SU_q(2))$. The functions $f\colon \NN\ni M\mapsto \tfrac{ d(M)^2}{\sum_{m=1}^{M} d(m)^2 } \in \RR$ and 
	$g\colon \NN\ni M\mapsto \tfrac{ d(M+1)^2}{\sum_{m=1}^{M} d(m)^2 } \in \RR$ are strictly decreasing.
\end{lemma}

\begin{proof}
	Fix $M\in\NN$. We have $f(M+1)<f(M)$ if, and only if
	\begin{equation}\label{eq11}\begin{split}
	d(M+1)^2\sum_{m=1}^{M} d(m)^2
	< d(M)^2 \sum_{m=1}^{M+1} d(m)^2.
	\end{split}\end{equation}
	To prove the above inequality, it is enough to argue that 
	\begin{equation}\label{eq12}
	d(M+1) d(k) \le  d(M) d(k+1) 
	\end{equation}
	holds for $k\in\{1,\dotsc, M\}$. Indeed, then inequality \eqref{eq11} holds as
	\begin{align*}
&\quad\;	d(M)^2  \sum_{m=1}^{M+1} d(m)^2 - 
	d(M+1)^2  \sum_{m=1}^{M} d(m)^2 \\&=
	d(M)^2  d(1)^2 +
	\sum_{m=1}^{M}  \bigl( d(M)^2d(m+1)^2 - d(M+1)^2 d(m)^2)>0.
	\end{align*}
	%But \eqref{eq2} is a consequence of the equality  
	%\[ d(M+1) d(k) - d(M) d(k+1) + d(M-k)d(1) = 0\]
	%(which is a special case of \cite[Formula 3]{Kachurik}).
	
	Recall then from the proof of Lemma \ref{dimensions} that we have $\alpha>0$ and $\beta <0$ such that $d(n) = \alpha q^{-n} + \beta q^n$, $n \in \NN$.
	Thus 
	\begin{align*} 	d(M+1) d(k) - d(M) d(k+1) &= \alpha \beta (q^{-M-1} q^k + q^{M+1} q^{-k} - q^M q^{-k-1}
	- q^{k+1} q^{-M}) 
	\\& = \alpha \beta ((q^{-M} q^k - q^M q^{-k})(q^{-1} - q))\le 0.
	\end{align*}
	
We can deal with the function $g$ very similarly. Note first that $g(M+1)<g(M)$ if and only if
\[
d(M+2)^2\sum_{m=1}^{M} d(m)^2
< d(M+1)^2 \sum_{m=1}^{M+1} d(m)^2.
\]
Then observe that
\begin{align*}
&\quad\;
d(M+1)^2  \sum_{m=1}^{M+1} d(m)^2- 
d(M+2)^2  \sum_{m=1}^{M} d(m)^2 \\
&=
d(M+1)^2  d(1)^2 +
\sum_{m=1}^{M}  \bigl( d(M+1)^2d(m+1)^2 - d(M+2)^2 d(m)^2\bigr),
\end{align*}	
and the last expression is strictly positive thanks to the inequality \eqref{eq12}, which we have already established.
	
	%One can also compute directly: the desired statement \eqref{eq2} is equivalent to
	%\[
	%(q^{-(M+1)} - q^{M+1}) (q^{-k} - q^k )  \le  (q^{-M} - q^M ) (q^{-(k+1)} - q^{k+1})
	%\]
	%and
	%\[
	%q^{-M-1-k} - q^{-M-1+k} - q^{M+1-k} + q^{M+1+k} \le 
	%q^{-M-k-1} - q^{-M+k+1} - q^{M-k-1} + q^{M+1+k}
	%\]
	%which follows from
	%\[\begin{split}
	%&\quad\;
	%q^{-M-1+k} + q^{M+1-k} - q^{-M+k+1} - q^{M-k-1}=
	%q^{-1}(q^{k-M} - q^{M-k}) - q(q^{k-M} - q^{M-k})\\
	%&=(q^{-1} - q)(q^{-(M-k)} - q^{M-k})\ge 0.
	%\end{split}\]
	
\end{proof}

We are ready to compute the  F\o lner constants in the case of the standard generating set.

\begin{lemma}\label{lemma8}
	Let $X=\{1\}$ be a generating set of $R(SU_q(2))$, $q \in (0,1)$. Then we have $\FolinnX(R(SU_q(2)))=1-q^2$, 
	$\Fol_X(R(SU_q(2)))=q^{-2} - q^{2}$.
\end{lemma}

\begin{proof}
	We begin with the inner F\o lner constant.
	By definition we have
	\[
	\FolinnX(R(SU_q(2)))=\inf_A \tfrac{|\partial_X^{\inn} A| }{|A|}
	\]
	where $A\subseteq \ZZ_+$ is finite and non-empty. So, fix such a set $A$, put $M=\sup A$ and define  $A'=\{0,\dotsc ,M\}$. Clearly we have $|A|\le |A'|$. Furthermore,
	\[
	\partial_X^{\inn} A'=
	\{ \alpha\in A'\mid
(\alpha-1 \in \ZZ_+\setminus A' )\,\vee \,(\alpha+1 \in \ZZ_+ \setminus A')	\}=
	\{M\}.
	\]
	Since $M\in A$ and $M+1\notin A$ we also have $M\in \partial_X^{\inn} A$, thus $|\partial_X^{\inn} A'| \le |\partial_X^{\inn} A|$. It follows that
	\[
	\frac{ |\partial_X^{\inn} A'|}{|A'|}\le  \frac{ |\partial_X^{\inn} A|}{|A|}
	\]
	and consequently
	\[
	\FolinnX(R(SU_q(2)))=\inf_{M\in \ZZ_+} \frac{ |\partial_X^{\inn} \{0,\dotsc, M\}|}{|\{0,\dotsc, M\}|}=\inf_{M\in \ZZ_+} \frac{ |\{M\}|}{|\{0,\dotsc, M\}|}.
	\]
The proof of Lemma \ref{dimensions} implies that we have $\alpha >0$ and $\beta<0$ such that  $d(n) = \alpha q^{-n} + \beta q^n$, $n \in \NN$. Since $	\tfrac{|\{M\}|}{|\{0,\dotsc, M\}|}=
	\tfrac{d(M)^2}{\sum_{k=0}^{M} d(k)^2}$, we can use the monotonicity of the function $f$ from Lemma \ref{lemma7} to get that
	\begin{align*}
	\FolinnX(R(SU_q(2)))&=
	\lim_{M\to\infty}
	\tfrac{|\{M\}|}{|\{0,\dotsc, M\}|}=
	\lim_{M\to\infty}
	\tfrac{d(M)^2}{\sum_{k=0}^{M} d(k)^2} =
	\lim_{M\to\infty} \frac{(\alpha q^{-M} + \beta q^{M})^2}
	{1+ \sum_{k=1}^{M} (\alpha q^{-k}+ \beta q^k )^2 }
	\\
	&=
	\lim_{M\to\infty} \frac{\alpha^2 q^{-2M} + 2\alpha \beta  + \beta^2 q^{2M} 
	}
	{1 + \sum_{k=1}^{M}( \alpha^2 q^{-2k} +2\alpha \beta  +\beta^2 q^{2k} )}= \frac{\alpha^2}
	{-\alpha^2 q^{-2}(1-q^{-2})^{-1}
	} \\
&=1-q^2.
	\end{align*}
	
The argument concerning the other uniform F\o lner constant is very similar. Given $A\subseteq \ZZ_+$  finite and non-empty, we note as above that for $M=\sup A$ and $A'=\{0,\dotsc ,M\}$ we have $\partial_X A'=
	\{M, M+1\}$ and 
	\[
	\Fol_X(R(SU_q(2)))=\inf_{M\in \ZZ_+} \frac{ |\partial_X^{\inn} \{0,\dotsc, M\}|}{|\{0,\dotsc, M\}|}=\inf_{M\in \ZZ_+} \frac{ |\{M, M+1\}|}{|\{0,\dotsc, M\}|}.
	\]
Hence, this time using the monotonicity of $f+g$, where $f$ and $g$ are functions of Lemma \ref{lemma7}, 
	\begin{align*}
	\Fol_X(R(SU_q(2)))&=
	\lim_{M\to\infty}
	\tfrac{|\{M, M+1\}|}{|\{0,\dotsc, M\}|}=
	\lim_{M\to\infty}
	\tfrac{d(M)^2 +d(M+1)^2}{\sum_{k=0}^{M} d(k)^2} \\&=
	\lim_{M\to\infty} \frac{(\alpha q^{-M} + \beta q^{M})^2 + (\alpha q^{-(M+1)} + \beta q^{M+1})^2}
	{1+ \sum_{k=1}^{M} (\alpha q^{-k}+ \beta q^k )^2 }
	\\
	&=
	\lim_{M\to\infty} \frac{\alpha^2(1 +q^{-2}) q^{-2M} + 4\alpha \beta  + \beta^2(1+q^2) q^{2M} 
	}
	{1 + \sum_{k=1}^{M}( \alpha^2 q^{-2k} +2\alpha \beta  +\beta^2 q^{2k} )}\\
	&= \frac{\alpha^2 (1+q^{-2})}
	{-\alpha^2 q^{-2}(1-q^{-2})^{-1}
	} =(1-q^2)(1+q^{-2}) = q^{-2} - q^{2}.
	\end{align*}
\end{proof}

As expected, the standard generating set turns out to be optimal.
%Now we can  prove the following result.

\begin{proposition}\label{prop12}
Let $q\in (0,1]$. Then 
\[\begin{split}
\Folinn(R(SU_q(2)))&=1-q^{2}, \\
\Fol(R(SU_q(2)))&=q^{-2} - q^{2}.
\end{split}\]
\end{proposition}

\begin{proof}
	If $q=1$ then  $R(SU(2))$ is amenable and the claim follows from Proposition \ref{prop10}. Assume thus that $q \in (0,1)$.
	
	Let $X$ be any generating set of $R(SU_q(2))$. The fusion rules of $SU_q(2)$ imply that there exists a representation $\alpha\in X$ which corresponds to an odd number in $\ZZ_+\simeq \Irr(SU_q(2))$. Furthermore, $\{\alpha\}$ is also a generating subset of $R(SU_q(2))$. As $\partial_{\{\alpha\}}^{\inn} A\subseteq \partial_X^{\inn} A$ and $\partial_{\{\alpha\}} A\subseteq \partial_X A$ for any finite $A\subseteq \Irr(SU_q(2))$, we can assume that $X=\{\alpha\}$. 
	
	Now, for any $m\in \ZZ_+$ we have
	\begin{equation}\label{eq31}
	m\otimes \alpha  \simeq |m-\alpha | \oplus (|m-\alpha|+2) \oplus\cdots\oplus
	(m+\alpha)
	\end{equation}
	hence $\partial_{\{1\}}^{\inn} A \subseteq \partial_{\{\alpha\}}^{\inn} A$,  $\partial_{\{1\}} A \subseteq \partial_{\{\alpha\}}A$, as $|m-1|$ and $m+1$ appear in the decomposition \eqref{eq31}. Consequently $\Folinn (R(SU_q(2)))=\Folinnj(R(SU_q(2)))$,
$\Fol (R(SU_q(2)))=\Fol_{\{1\}}(R(SU_q(2)))$ asdf and the result follows from Lemma \ref{lemma8}.
\end{proof}

Next, let us record the value of the uniform growth rate for $R(SU_q(2))$.

\begin{proposition}\label{prop14}
Let $q\in \left(0,1\right]$. Then $\omega(R(SU_q(2)))=q^{-2}$.
\end{proposition}

The above result is a special case of Theorem \ref{thm1}, where we compute the uniform growth rates for more general $q$-deformations. One can also obtain this value by more elementary means; we present such reasoning in the next section, in the case of $SO_q(3)$.\\

The following proposition computes the uniform Kazhdan constant for the regular representation of the fusion algebra 
$R(SU_q(2))$.

\begin{proposition} \label{prop11}
Let $q \in (0,1]$. Then $\Kaz(\rho,R(SU_q(2)))= 1 - \frac{2}{[2]_q}  =\tfrac{(1-q)^2}{q^2+1}$.
\end{proposition}
\begin{proof}
If $q=1$ then $SU(2)$ is coamenable and $\Kaz(\rho,R(SU(2)))=0$ follows from Proposition \ref{prop:KazhFol} and earlier discussions. Assume $q<1$.	
	
We will use the following standard result, which follows from the spectral theorem \cite[Section 4.3]{SoltanPrimer}: if $T$ is a bounded, self-adjoint operator on a Hilbert space $\mathcal{H}$ and $r\in \RR$, then
\begin{equation}\label{eq14}
\inf_{\xi\in \mathcal{H}, \|\xi\|=1} \|T\xi-r\xi\|=\oon{dist}(\sigma(T),r),
\end{equation}
where $\sigma(T)$ is the spectrum of $T$. 
%Indeed, writing $E_T$ for the spectral measure of $T$ we have by \cite[Theorem 10.9]{SoltanPrimer}
%\[
%\|T\xi-r\xi\|^2=
%\int_{\sigma(T)} |\lambda-r|^2 \md\, \la \xi | E_T \xi\ra (\lambda) \ge 
% \oon{dist}(\sigma(T),r)^2 
% \int_{\sigma(T)}1  \md\, \la \xi | E_T \xi\ra (\lambda)=
% \oon{dist}(\sigma(T),r)^2
%\]
%for any $\xi\in \mathcal{H}$ with $\|\xi\|=1$. Conversely, choose $\lambda_0\in \sigma(T)$ such that %$|\lambda_0-r|=\oon{dist}(\sigma(T),r)$. Then for any $\eps>0$, the spectral projection $1_{\left[ -\eps+\lambda_0,\lambda_0+\eps\right]}(T)$ is non-zero; choose vector $\xi$ of norm $1$ in its image. Then
%\[
%\|T\xi - r \xi \|^2=
%\int_{\sigma(T)}1_{\left[-\eps+\lambda_0,\lambda_0+\eps\right]}(\lambda) |\lambda-r|^2 \md\, \la \xi | %E_T \xi\ra (\lambda) \le 
%(|\lambda_0-r|+\eps)^2=
%( \oon{dist}(\sigma(T),r)
%+\eps)^2.
%\]
%Since $\eps>0$ was arbitrary, this proves \eqref{eq14}.\\

We begin by analysing properties of the regular representation $\rho$, acting on the Hilbert space $\ell^2(\ZZ_+)$. For clarity of the notation write $\alpha_m=m\in \Irr(SU_q(2))$ and $R=R(SU_q(2))$. Observe that
	\begin{equation}\label{eq15}
	\rho(\alpha_1)\delta_n=\begin{cases}
	\delta_1, & n=0, \\
	\delta_{n-1}+\delta_{n+1}, & n\ge 1,
	\end{cases}
	\end{equation}
	hence $\rho(\alpha_1)$ is the sum of unilateral shift and its adjoint, in particular $\rho(\alpha_1)^*=\rho(\alpha_1)$. Let us define von Neumann algebra
\[
	\M=\{\rho(\alpha_1)\}''=
	\ov{\lin}^{\,\swot}\{\rho(\alpha_1)^k\mid k\in \ZZ_+\}\subseteq \B(\ell^2(\ZZ_+))
	\]
	and a vector state $	\omega=\omega_{\delta_0}\in \M_*$. Equation \eqref{eq15} implies that $\delta_0$ is a cyclic vector for $\M$, and since $\M$ is commutative, $\delta_0$ is also separating. Consequently, $\omega$ is a faithful normal state. Observe that for $k\ge 0$ the number
	\[
	\omega(\rho(\alpha_1)^k)=
	\omega(\rho(\alpha_1^{\otimes\, k}))=
	\la \delta_0 | \rho(\alpha_1^{\otimes k}) \delta_0\ra 
	\]
	is equal to the multiplicity of the trivial representation in $\alpha_1^{\otimes k}$, hence
	\begin{equation}\label{eq13}
	\omega(\rho(\alpha_1)^k)=h(\chi(\alpha_1^{\otimes k}))=
	h(\chi(\alpha_1)^k),
	\end{equation}
where $h$ is the Haar integral and $\chi(U)$ is the character of representation $U$, see \cite{NeshveyevTuset}. Let $P\colon \RR\rightarrow \CC$ be a polynomial function, $P=\sum_{k=0}^{n} c_k \boldsymbol{t}^k$. By \cite[Proposition 6.2.11]{Timmermann} and \eqref{eq13} we have
	\[
	\omega ( P(\rho(\alpha_1)))=
	\sum_{k=0}^{n} c_k \omega( \rho(\alpha_1)^k)=
	\sum_{k=0}^{n} c_k h( \chi(\alpha_1)^{k})=
	h(P(\chi(\alpha_1)))=
	\tfrac{1}{2\pi} \int_{-2}^{2} P(t) \sqrt{4-t^2} \md t.
	\]
By continuity we obtain $\omega ( f(\rho(\alpha_1)))=
	\tfrac{1}{2\pi} \int_{-2}^{2} f(t) \sqrt{4-t^2} \md t$ for any $f\in \mrm{C}_0(\RR)$, and it follows that $\sigma(\rho(\alpha_1))=\left[-2,2\right]$. \\

In what follows, we will use Chebyshev polynomials of the second kind. More precisely, let $U_n(x)\,(n\in\ZZ_+)$ be the Chebyshev polynomials of the second kind (see \cite[Chapter 18]{NIST}), and let $P_n(x)=U_n(\tfrac{x}{2})$ be their rescaled versions. We have $P_0(x)=1, P_1(x)=x$ and $x P_n(x)=P_{n+1}(x)+P_{n-1}(x)\,(n\in\NN)$ \cite[Section 18.5, 18.9]{NIST}. Comparing this recurrence relation with the fusion rule of $SU_q(2)$ (equation \eqref{eq16}), we see that $\alpha_m=P_m(\alpha_1)$ in $R(SU_q(2))$.\\

Consider now an arbitrary generating set $X\subseteq \ZZ_+=\Irr(SU_q(2))$. By definition of the Kazhdan constants, spectral mapping theorem and formula \eqref{eq14}, we have
\[\begin{split}
\Kaz(X,\rho,R)&=\inf_{\xi\in \ell^2(\ZZ_+), \|\xi\|=1} \max_{m\in X} \tfrac{ \|\rho(\alpha_m) \xi - d(m)\xi \|}{d(m)}\ge 
 \max_{m\in X} \inf_{\xi\in \ell^2(\ZZ_+), \|\xi\|=1}\tfrac{ \|\rho(\alpha_m) \xi - d(m)\xi \|}{d(m)}\\
 &=
  \max_{m\in X} \tfrac{1}{d(m)}
  \oon{dist}\bigl(\sigma(\rho(\alpha_m) ) , d(m)\bigr)=
  \max_{m\in X} \tfrac{1}{d(m)}
  \oon{dist}\bigl(P_m(\sigma(\rho(\alpha_1) )) , d(m)\bigr)\\
  &=
  \max_{m\in X} \tfrac{1}{d(m)}
  \oon{dist}\bigl(P_m(\left[-2,2\right]) , d(m)\bigr).
 \end{split}\]

Using \cite[Section 18.7, 18.14]{NIST} we see $|P_m(x)|\le m+1=P_m(2)\,(m\in\ZZ_+, -2\le x \le 2)$. Consequently,
\begin{equation}\label{eq17}
\Kaz(X,\rho,R)\ge 
\max_{m\in X} \tfrac{d(m) -m-1}{d(m)} =
\max_{m\in X} \bigl( 1- \tfrac{m+1}{[m+1]_q}\bigr).
\end{equation}
Any generating set must contain a non-zero $m \in \mathbb{Z}_+$. Thus Lemma \ref{lemma9} implies 
 \[
\Kaz(X,\rho,R) \ge 1-\tfrac{2}{[2]_q}=
\frac{q+q^{-1}-2}{q+q^{-1}}=
\frac{q^2+1-2q}{q^2+1}=
\tfrac{(1-q)^2}{q^2+1}.
\]
Repeating the reasoning above, we see that this lower bound is attained for the generating set $X=\{1\}$, which ends the proof.
\end{proof}

Note that we in particular obtain for $R=R(SU_q(2))$, $q \in (0,1)$,
\[
 \Kaz(\rho,R)^2=\bigl(\tfrac{q+q^{-1}-2}{q+q^{-1}}\bigr)^2=
1-\tfrac{4q}{1+q^2} + \tfrac{4 q^2}{(1+q^2)^2}<1-q^2 = \Fol(R),
\]
which shows that the inequality in Proposition \ref{prop:KazhFol} can be sharp.
%Indeed, alternatively
%\[
%0<1-q^2- \tfrac{1}{2} (1-\tfrac{4q}{1+q^2} + \tfrac{4 q^2}{(1+q^2)^2})=
%(1-q)\tfrac{1+5q+q^2 +5q^3+2q^4+2q^5 }{2(1+q^2)^2}
%\]
%which is clearly true.

\medskip

Fix a matrix $F\in GL_N(\CC)\,(N \ge 2)$ such that $F \ov{F}\in \{\pm\I\}$. Let $O_F^+$ be the associated free orthogonal quantum group \cite[Example 1.1.7]{NeshveyevTuset}. It was shown in \cite{BanicaOrth} that $R(O_F^+) = R(SU_q(2))$, for a parameter $q\in (0,1]$ uniquely determined by the matrix $F$ via $q+q^{-1}=\Tr(F^*F)$.
%Indeed, by \cite[Remarks 1.5]{DaeleWang} we have $(S^2\otimes \id)(u) (\I\otimes (F^*F)^t)=(\I\otimes (F^*F)^t)u$, where $u$ is the fundamental representation of $O_F^+$. But $(S^2\otimes \id)(u)=u^{cc}$ hence $(F^*F)^t=\uprho_u$, as $F\ov{F}\in \{\pm \I\}$ implies $\Tr((F^*F)^{-1})=\Tr(F^*F)$.
%Alternatively: $u^c=(\I\otimes F^{-1}) u (\I\otimes F)$ hence $u^{cc}=(\I\otimes F^t) u^c (\I\otimes F^{t -1})=(\I\otimes F^t F^{-1}) u (\I\otimes FF^{t -1})$ hence $F^t F^{-1}\sim \uprho_u$ but $F^tF^{-1}=\pm F^t \ov{F}=\pm (F^*F)^t$.
In particular, considering the free orthogonal quantum groups $O_N^+$ (by taking $F=\I\in GL_N(\CC)$), we obtain the following corollary.

%\begin{corollary} If $N\in \NN, N \geq 3$, then for the corresponding free orthogonal quantum group we obtain
%	\[ \Folinn(R(O_N^+)) = 1 - 4 (N+%\sqrt{N^2-4})^{-2},\]
%	\[ \Fol(R(O_N^+)) = 1 - 4 (N+\sqrt{N^2-4})^{-2}. ????\]
%\end{corollary}

\begin{corollary} If $N\in \NN, N \geq 2$, then for the corresponding free orthogonal quantum group we obtain
\[\begin{split}
\Folinn(R(O_N^+)) &= 1 - \tfrac{1}{4} (N-\sqrt{N^2-4})^{2},\\
\Fol(R(O_N^+)) &=
\frac{4}{(N-\sqrt{N^2-4}\,)^{2}}
- \tfrac{1}{4} (N-\sqrt{N^2-4})^{2}, \\
\omega(R(O_N^+)) & = \frac{4}{(N-\sqrt{N^2-4}\,)^{2}}  , \\
\Kaz(\rho,R(O_N^+)) & = 1- \frac{2}{N}. 
\end{split}\]
\end{corollary}

\section{Fusion algebra of $SO_q(3)$}
\label{Sec:SO}

In this section we will treat the fusion algebra associated with the quantum group $SO_q(3)$, with $q \in (0,1]$. This happens to be also, with a suitably chosen dimension function, the fusion algebra associated with the quantum permutation groups. The arguments will roughly follow those of the last section, so we will often simply  indicate  computational differences.

Once again we have $I=\ZZ_+$, $0=e$ and the fusion algebra $R(SO_q(3))$ is generated by $1$, with the fusion rules given this time by
\[ m \otimes n = |m-n| \oplus (|m-n|+1) \oplus \cdots \oplus (m+n), \;\;\; m,n \in \ZZ_+. \]
An arbitrary dimension function $\tilde{d}\colon \ZZ_+\to [1, \infty)$ compatible with the above fusion rules is thus determined by $\tilde{d}(1)\geq 3$ (as before, the lower bound follows from Proposition \ref{prop:amenmin}); for the quantum dimension of $SO_q(3)$ with $q \in (0,1)$ we have specifically $d(1)=[3]_q =\tfrac{q^3 - q^{-3}}{q- q^{-1}}= q^2 + q^{-2}+1$. Again we can realise all possible dimension functions on the ring above using some $q \in (0,1]$. In fact, we can study $R(SO_q(3))$ as a sub-fusion algebra of $R(SU_q(2))$ generated by `even' representations of $SU_q(2)$. This is clear from the fusion rules above and the next lemma.

Record in particular the recurrence formula:
\[ d(n) d(1) = d(n+1) + d(n) + d(n-1), \;\; n \in \NN,\]
which implies the following version of Lemma \ref{dimensions}.

%{\color{red} The next lemma has to be corrected!}
%\begin{lemma} \label{dimensions2} FALSE!!!%
%	Define $0<q\leq 1$ via $d(1)=q+q^{-1} - 1$. Then  $d(n)=[n+1]_q$ for all $n\in\ZZ_+$.
%\end{lemma}
%\begin{proof}
%	This just follows as previously, as the recurrence relation becomes exactly the same as the one before if we replace $d_1$ by $d_1-1$.
%\end{proof}

%{\color{red} Here is a first go.}

\begin{lemma} \label{dimensions2} 
 Let $q \in (0,1)$ and let $d\colon \ZZ_+\to [1,\infty)$ denote the quantum dimension function of $R(SO_q(3))$.  
	Then 
	\[ d(n) = [2n+1]_q = \alpha q^{-2n} + \beta q^{2n}, \;\;\; n \in \ZZ_+,\]
	where $\alpha, \beta \in \mathbb{R}$ are determined by 
	\[ \alpha  + \beta  = 1, \;\;\; \alpha q^{-2} + \beta q^2 = [3]_q.  \]
	In particular one can show that $\alpha>0$, $\beta <0$.
	%Define $0<q\leq 1$ via $d(1)=q+q^{-1} - 1$. Then  $d(n)=[n+1]_q$ for all $n\in\ZZ_+$.
\end{lemma}

%{\color{teal}Do we really need $d$ and $\tilde{q}$ here?}

\begin{proof}
Set $\tilde{q} =q^2$, $\tilde{d}=d(1)-1$.
Note that  $\tilde{d}=\tilde{q}+\tilde{q}^{-1}$ . %and let 
	%\[ d_{\pm} = \frac{d \pm \sqrt{d^2-4}}{2}.\]
The values $d(n)$ are determined by the recurrence relation 
	\[ \tilde{d} d(n) = d(n+1) + d(n-1), \]
	valid for all $n \in \NN$. %The standard trick with writing the recurrence relation shows that for each $n \in \N$ we have 
	%\[ \begin{pmatrix}  d(n+1)\\ d(n) \end{pmatrix} =  \begin{pmatrix}  d & -1\\ 1 & 0 \end{pmatrix} \begin{pmatrix}  d(n)\\ d(n-1) \end{pmatrix},\]
	%and it is easy to check that the matrix in the middle has eigenvalues $\{d_+, d_-\}$ (say $d>2$).
	As in the proof of Lemma \ref{dimensions} we deduce the existence of $\alpha, \beta$ as above (note that the `recurrence matrix' is identical, only the initial conditions change).
	
Alternatively, we can simply observe that if we set $d(n) = [2n+1]_q$, $n \in \ZZ_+$ we obtain the coefficients which satisfy both the required initial and recurrence conditions. 
	
Either description shows that $\alpha = \frac{q^{-1}}{q^{-1} - q}$, $\beta = - \frac{q}{{q}^{-1} - q}$, so that the last statement follows. 
	
\end{proof}

Note that  the analogous formula for the dimensions in the proof of \cite[Theorem 4.1(3)]{BanicaVergnioux} contains a typo.

\begin{lemma}
	Let $q \in (0,1)$ and let $d\colon \ZZ_+\to [1,\infty)$ denote the quantum dimension function of $R(SO_q(3))$. The functions  $f\colon\NN \ni M \mapsto \frac{d(M)^2}{\sum_{m=1}^M d(m)^2} \in \mathbb{R}$  and $g\colon\NN \ni M \mapsto \frac{d(M+1)^2}{\sum_{m=1}^M d(m)^2} \in \mathbb{R}$ are strictly decreasing.	
\end{lemma}
\begin{proof}
	Identical as the proof of Lemma \ref{lemma7}, as the function $d$ has the same form (c.f.\ Lemma \ref{dimensions} and Lemma \ref{dimensions2}).
\end{proof}

%Thus the same proof as before gives the following result.

\begin{proposition}
 Let $q\in (0,1]$. Then 
\[\begin{split}
 \Folinn(R(SO_q(3)))&=1-q^{4},\\
\Fol(R(SO_q(3)))&=q^{-4} - q^{4}.
\end{split}\]
\end{proposition}
\begin{proof}
For $q=1$ the claim follows by amenability.	
	
If $q \in (0,1)$ we first pass to $\tilde{q} = q^2$ and then use the last two lemmas to repeat the same arguments  as in the proofs of Lemma \ref{lemma8} and Proposition \ref{prop12}. Note that the specific form of constants $\alpha, \beta$ plays no role in the latter proofs.
\end{proof}

In this case it is also very easy to give an elementary computation  of the uniform growth rate.

\begin{proposition}\label{prop15}
Let $q\in (0,1]$. Then  $\omega(R(SO_q(3)))=q^{-4}$.	
\end{proposition}
\begin{proof}
Let $X \subseteq \ZZ_+$ be a finite generating set for $R=R(SO_q(3))$ and set $D_X = \max\{m \in X\} \in \NN$. 
It is easy to deduce from the fusion rules that for any $n \in \NN$ we have
\[B_X(n)\supseteq B_{\{1\}}(D_X n).\]
This implies that 
\[ \omega_X(R) = \lim_{n \in \NN} \sqrt[n]{|B_X(n)|} \geq \lim_{n \in \NN} \sqrt[n]{|B_{\{1\}}(D_X n)|} = 
\omega_{\{1\}}(R)^{D_X}.\]	
Thus $\omega(R) = \omega_{\{1\}}(R)$. But $S_{\{1\}}(n) = n$ for every $n \in \NN$, so that by Lemma \ref{dimensions2}  for $q\in (0,1)$ we have
\[ \lim_{n \in\NN} \sqrt[n]{|S_{\{1\}}(n)|}= \lim_{n \in\NN} \sqrt[n]{[2n+1]_q^2}
= \lim_{n \in\NN} \left(\sqrt[n]{\frac{q^{-2n+1} - q^{2n+1}}{q^{-1}- q}}\right)^2 = q^{-4}. \]	
Lemma \ref{lemma3} ends the proof (the case of $q=1$ follows similarly).	

\end{proof}

Finally we compute the uniform Kazhdan constant for the regular representation.

\begin{proposition} \label{prop:KazhdanSO}
	Let $q \in (0,1]$. Then $\Kaz(\rho,R(SO_q(3)))= 1 - \frac{3}{[3]_q} = 1-\tfrac{3}{q^2 +q^{-2}+1}$.
\end{proposition}

\begin{proof}
	If $q=1$, then $SO(3)$ is coamenable and $\Kaz(\rho,R(SO(3)))=0$ follows from Proposition \ref{prop:KazhFol} and earlier discussions.

Fix $q\in (0,1)$, and to avoid confusion, write $\rho_{SO_q(3)}$ and $\rho_{SU_q(2)}$ for the right regular representations of respectively, $R(SO_q(3))$ and $R(SU_q(2))$. We will also write $\alpha_m=m\in \ZZ_+=\Irr(SU_q(2))$ and $\beta_m=m\in \ZZ_+=\Irr(SO_q(3))$. Note that $R(SO_q(3))$ is a sub-fusion algebra of $R(SU_q(2))$ via $\beta_m\mapsto \alpha_{2m}$, $m \in \ZZ_+$. Next, consider the von Neumann subalgebra
\[
\mc{A}=\{\rho_{SU_q(2)} (\alpha_{2m})\mid m\in\ZZ_+\}''\subseteq 
\rho_{SU_q(2)} (R(SU_q(2)))''
\]
and observe that $\mc{H}=\ov{\lin}\{\delta_{2m}\mid m\in\ZZ_+\}\subseteq \ell^2(\Irr(SU_q(2)) )$ is an invariant subspace for $\mc{A}$. Unitary isomorphism $\mc{H}\ni \delta_{2m}\mapsto \delta_{m}\in \ell^2(\Irr(SO_q(3)) )$ and the fusion rules imply that $\mc{A}\restriction_{\mc{H}}$ is isomorphic to $\rho_{SO_q(3)}(R(SO_q(3)))''$ via $\rho_{SU_q(2)}(\alpha_{2m})\restriction_{\mc{H}}\mapsto \rho_{SO_q(3)}(\beta_m)$ for $m\in\ZZ_+$. The restriction mapping $\mc{A}\ni a\mapsto a\restriction_{\mc{H}}\in \mc{A}\restriction_{\mc{H}}$ is also an isomorphism. This follows from the fact that $\nu=\omega_{\delta_0}\in(\mc{A}\restriction_{\mc{H}})_*$ satisfies $\nu(a\restriction_{\mc{H}})=\omega(a)$, where $\omega=\omega_{\delta_0}\in \mc{A}_*$ is a faithful normal state (see the proof of Proposition \ref{prop11}). Combining these properties with the fact that spectrum of an element $x$ in a unital \cst-algebra is equal to its spectrum in any unital subalgebra containing $x$ (\cite[Corollary I.5.7]{Davidson}), we see that the spectrum of $\rho_{SU_q(2)}(\alpha_{2m})$ in $\rho_{SU_q(2)}(R(SU_q(2)))''$ is equal to the spectrum of $\rho_{SO_q(3)}(\beta_m)$ in $\rho_{SO_q(3)}(R(SO_q(3)))''$.\\

Now, let $X\subseteq \ZZ_+=\Irr(SO_q(3))$ be a generating set. Arguing as in the proof of Proposition \ref{prop11}, we calculate
\begin{align*}
\Kaz\bigl(X,\rho&_{SO_q(3)},R(SO_q(3))\bigr)=
\inf_{\xi\in \ell^2(\ZZ_+), \|\xi\|=1}
\max_{m\in X} \tfrac{ \|\rho_{SO_q(3)}(\beta_m) \xi-d(\beta_m)\xi\|}{d(\beta_m)}\\
&\ge 
\max_{m\in X}\inf_{\xi\in \ell^2(\ZZ_+), \|\xi\|=1}\tfrac{ \|\rho_{SO_q(3)}(\beta_m) \xi-d(\beta_m)\xi\|}{d(\beta_m)}=
\max_{m\in X}\tfrac{1}{d(\beta_m)}
\oon{dist}\bigl( \sigma ( \rho_{SO_q(3)}(\beta_m) ), 
d(\beta_m)\bigr)\\
&=
\max_{m\in X}\tfrac{1}{d(\alpha_{2m})}
\oon{dist}\bigl( \sigma ( \rho_{SU_q(2)}(\alpha_{2m}) ), 
d(\alpha_{2m})\bigr)=
\max_{m\in X}
\bigl(1- \tfrac{2m+1}{[2m+1]_q}\bigr) \ge 1-\tfrac{3}{[3]_q}.
\end{align*}
Similarly to the case of $SU_q(2)$, we see that this lower bound is attained for $X=\{\beta_1\}$, so that
\[
\Kaz\bigl(\rho_{SO_q(3)},R(SO_q(3))\bigr)=
1-\tfrac{3}{[3]_q}=1-
\tfrac{3}{q^2+q^{-2}+1}.
\]
\end{proof}

One should note that the fusion ring studied in this section can be also realised as the one related to the TLJ $\mc{G}^{\lambda}$ category, i.e.\ the category of bimodules related to  an extremal inclusion  $\N\subseteq \M$ of type $\oon{II}_1$ factors with finite index $[\M\colon \N]=\lambda^{-1}\ge 4$ and principal graph $A_{\infty}$ (with $\lambda^{-\frac{1}{2}} = q+q^{-1}$; see the beginning of the proof of \cite[Theorem 7.1]{PopaVaes}), or to the quantum permutation groups $S_N^+$ for $N\geq 4$ (see \cite{BBC}). In particular we obtain the following corollary.

%{\color{red} Again: add some explicit comments on constants for $S_N^+$.}

\begin{corollary} If $N\in \NN, N \geq 4$, then for the corresponding quantum permutation group we obtain
	\[\begin{split}
	\Folinn(R(S_N^+)) &= 1 - \tfrac{(N- 2 -\sqrt{N^2-4N})^{2}}{4} ,\\
	\Fol(R(S_N^+)) &= 
	\frac{4}{(N- 2 -\sqrt{N^2-4N})^{2}}
	- \tfrac{(N- 2 -\sqrt{N^2-4N})^{2}}{4}, \\
	\omega(R(S_N^+)) & = 	\frac{4}{(N- 2 -\sqrt{N^2-4N})^{2}}  , \\
	\Kaz(\rho,R(S_N^+)) & = 1- \frac{3}{N-1}. 
	\end{split}\]
\end{corollary}

\section{Growth rate for fusion algebras of general $q$-deformations}\label{sec:qdef}

Let $G$ be a compact Lie group, which we assume to be simply connected and semisimple. $G$ comes together with a number of auxilliary objects: complexified Lie algebra $\mf{g}$, maximal torus $T\subseteq G$, its complexified Lie algebra $\mf{h}$, weight lattice $\bf{P}\subseteq\mf{h}^*$ and its positive cone $\bf{P}^+$, root system $\Phi\subseteq \bf{P}$, positive roots $\Phi^+\subseteq \Phi$, simple roots $\{\alpha_1,\dotsc,\alpha_r\}\subseteq \Phi^+$, fundamental weights $\{\varpi_1,\dotsc,\varpi_r\}\subseteq \bf{P}^+$, Weyl vector $\rho\in \bf{P}^+$, Weyl group $W$ and the longest element $w_\circ\in W$. We equip $\mf{h}^*$ with the $W$-invariant inner product given by the rescaling of the Killing form on $\mf{g}$; we follow the usual convention and normalise it so that $\la \alpha|\alpha\ra=2$ for short roots $\alpha\in\Phi$. We have $\la \varpi_i|\alpha_j\ra = \tfrac{\la \alpha_j|\alpha_j\ra}{2} \delta_{i,j}$ for $1\le i,j\le r$. The number $r$ is called the rank of $\Phi$, and by the root space decomposition (\cite[Section II.8]{Humphreys}) we have $\dim(G)=r+s$, where $s=\# \Phi$ is the number of roots. We will need the well known facts that $w_\circ\Phi^+=-\Phi^+$, $w_\circ=w_\circ^{-1}$ and $w_\circ \rho=-\rho$ (\cite[Exercise 20.2]{Bump}).

Fix $0<q<1$. With $G$ and $q$ one associates compact quantum group $G_q$, which can be thought of as a $q$-deformation of the classical Lie group $G$ (for the construction see\footnote{One should be aware that certain differences in conventions appear in the literature. This  will however not  play any role in our work.} \cite[Section 2.4]{NeshveyevTuset} or \cite[Section 5.3]{DeCommer}, \cite[Section 4]{KrajczokSoltan}). As in the classical case, irreducible representations of $G_q$ are indexed by ${\bf{P}}^+$. The representation theory of the  quantum group $G_q$ has the same classical dimension function,  fusion rules and conjugacy operation as that of $G$. In particular  $\ov{\lambda}\simeq -w_\circ \lambda$ for $\lambda\in \bf{P}^+$. For $\lambda\in \bf{P}^+$, or more generally for any representation $\lambda$, we will denote by $\Pi(\lambda)\subseteq \bf{P}$ the set of weights of $\lambda$. It is invariant under the action of $W$ (\cite[Chapter 21]{Bump}).

In this section we will compute the uniform growth rate $\omega(R(G_q))$ for arbitrary $q$-deformation (see Remark \ref{rem1} and Theorem \ref{thm1}), where as usual we equip the fusion algebra $R(G_q)$ with the quantum dimension function.  We can use this together with Proposition \ref{prop1} to arrive at an upper bound on the F{\o}lner constant $\Folinn(R(G_q))$. 

For the classical dimension function we have $\omega_X(R(G_q),\dim)=1$ for any finite generating set $X$, hence $\omega(R(G_q),\dim)=1$; see \cite[Theorem 2.1]{BanicaVergnioux} and the following lemma. 

\begin{lemma}\label{lemma4}
Let $X\subseteq \bf{P}^+$ be a finite generating set for $R(G_q)$. The function $\ZZ_+\ni n\mapsto \# B_X(n)\in \NN$ has polynomial growth.
\end{lemma}

\begin{proof}
See the proof of \cite[Theorem 2.1]{BanicaVergnioux} for a particular choice of $X$; the general case follows from an easy standard argument, going back already to \cite[Lemma 1]{Milnor}.
\end{proof}

Recall from Section \ref{sec:discrete/compact} the notion of $q$-numbers $[x]_q$ %=\frac{q^{-x}-q^x}{q^{-1}-q}\,
$(x\in \CC,0<q<1)$. For the calculation of growth we need a way of calculating the quantum dimension $d(\lambda)$ of $\lambda\in \bf{P}^+$. First we recall a precise expression, which resembles the classical Weyl dimension formula (\cite[Theorem 22.4]{Bump}). For the proof we refer to \cite[Lemma 1]{ZhangGouldBracken}.

\begin{proposition} \label{prop6}
For $\lambda\in {\bf{P}}^+$ we have $d (\lambda)=\prod_{\alpha\in\Phi^+}\frac{ [\la \lambda+\rho |\alpha\ra ]_q}{[\la\rho|\alpha\ra ]_q}$.
\end{proposition}

While the above formula  gives us an exact answer, it is not always easy to use in practice. In Proposition \ref{prop5} we will use it to derive a more practical bound on the quantum dimension.

\begin{lemma}\label{lemma6}
Let $\lambda\in {\bf{P}^+}$ and $\Pi(\lambda)\subseteq \bf{P}$ be the set of weights of $\lambda$. Recall that $\rho$ denotes the Weyl vector. We have
\[
\max_{\mu\in \Pi(\lambda)} \la \mu|\rho\ra=\la \lambda|\rho\ra \quad\textnormal{and}\quad
 \min_{\mu\in \Pi(\lambda)} \la \mu|\rho\ra=-\la \lambda|\rho\ra .
\]
\end{lemma}

\begin{proof}
Take any $\mu\in \Pi(\lambda)$. Then $\lambda-\mu$ is a sum of positive roots (\cite[Proposition 21.3]{Humphreys}), say $\lambda-\mu=\sum_{k=1}^{K} \beta_k$ with $\beta_k\in \Phi^+$. Since $\rho$ belongs to the positive Weyl chamber, we have $\la\beta_k|\rho\ra\ge 0$ and consequently
\[
\la \mu |\rho\ra =
\la \lambda|\rho\ra - \sum_{k=1}^{K} \la \beta_k |\rho \ra \le 
\la \lambda|\rho\ra.
\]
Since $\lambda\in\Pi(\lambda)$, this ends the proof of the first part. For the second, take again any $\mu\in \Pi(\lambda)$ and write $\lambda-w_\circ \mu=\sum_{l=1}^{L}\gamma_l$ with $\gamma_l\in \Phi^+$. Then $w_\circ\lambda- \mu=\sum_{l=1}^{L}w_\circ \gamma_l$ and
\[
\la \mu|\rho\ra= \la 
w_\circ\lambda |\rho\ra - \sum_{l=1}^{L} 
\la w_\circ\gamma_l|\rho\ra=
-\la 
\lambda |\rho\ra + \sum_{l=1}^{L} 
\la\gamma_l|\rho\ra\ge -\la \lambda|\rho\ra.
\]
Since $w_\circ\lambda\in \Pi(\lambda)$ and $\la w_\circ\lambda|\rho\ra=-\la\lambda|\rho\ra$, this ends the proof.
\end{proof}

Recall that $s=\#\Phi=\dim(G)-r$ is the number of roots. The parameter $\Gamma(\alpha)$ for an irreducible representation $\alpha$ of a compact quantum group was introduced before Proposition \ref{prop9}

\begin{proposition}\label{prop5}
For $\lambda\in \bf{P}^+$ we have $\Gamma(\lambda)=q^{-\la \lambda |2\rho\ra}$ and
\[
q^{-\la \lambda|2 \rho\ra }\le 
d(\lambda)\le 
(q^{-1}-q)^{-s/2} q^{-\la \rho|2 \rho\ra } \,
q^{- \la \lambda|2 \rho\ra }.
\]
\end{proposition}

\begin{proof}
To see the lower bound on $d(\lambda)$, recall from \cite[Lemma 4.4]{KrajczokSoltan} that $\uprho_{\lambda}$ is equal to the image of $K_{2\rho}$ under the corresponding representation $U_q\mf{g}\rightarrow \B(\msf{H}_\lambda)$. This operator is diagonal with eigenvalues $q^{\la \mu | 2\rho\ra}\,(\mu\in\Pi(\lambda))$, hence $\|\uprho_\lambda\|=\max_{\mu\in\Pi(\lambda)}q^{\la \mu| 2\rho\ra}=q^{-\la\lambda|2\rho\ra}$ (Lemma \ref{lemma6}). The lower bound then follows from $\|\uprho_\lambda\|\le \Tr(\uprho_\lambda)=d(\lambda)$. This also proves the first claim, as $\Gamma(\lambda)=\|\uprho_\lambda\|$.

We will use Proposition \ref{prop6} to bound the quantum dimension of $\lambda$ from above. Fix $\alpha\in \Phi^+$ and write $\alpha=\sum_{i=1}^{r} C(\alpha,i)\alpha_i$ for some $C(\alpha,i)\in \ZZ_+$. Using $\rho=\sum_{j=1}^{r}\varpi_j$ and our normalisation of the inner product, we calculate 
\[
\la\rho|\alpha\ra =\sum_{i,j=1}^{r}
\la \varpi_j|C(\alpha,i)\alpha_i\ra =
\sum_{i=1}^{r} C(\alpha,i) \tfrac{\la \alpha_i|\alpha_i\ra }{2} \ge 1.
\]
It follows that $\prod_{\alpha\in\Phi^+} [\la \rho|\alpha\ra]_q \ge \prod_{\alpha\in\Phi^+} [1]_q =1$. On the other hand, since $\rho=\tfrac{1}{2}\sum_{\alpha\in\Phi^+}\alpha$, we have
\[
\prod_{\alpha\in\Phi^+} [\la \lambda+\rho|\alpha\ra]_q \le 
\prod_{\alpha\in\Phi^+} \tfrac{q^{-\la \lambda+\rho|\alpha\ra }}{q^{-1}-q}=
\tfrac{q^{- \la \lambda+\rho|2 \rho\ra }}{(q^{-1}-q)^{s/2}}.
\]
Proposition \ref{prop6} ends the proof.
\end{proof}

As a first step towards calculation of the uniform growth rate $\omega(R(G_q))$, we provide a general formula for the growth rate for an arbitrary generating set.

\begin{proposition}\label{prop7}
Let $Y\subseteq \bf{P}^+$ be a finite generating set for $R(G_q)$. Then
\[
\omega_Y(R(G_q)) = 
\max_{\lambda\in Y} q^{-4\la\lambda|\rho\ra}
\]
\end{proposition}

\begin{proof}
Take $n\in\NN$. Observe that for $0\neq \lambda\in B_Y(n)$ we have $ \lambda\subseteq \lambda_1\otimes\cdots\otimes \lambda_k$ for some $1\le k \le n, \lambda_1,\dotsc,\lambda_k\in Y$. Hence
\[
\lambda\in \Pi(\lambda_1\otimes\cdots\otimes\lambda_k)=\{\mu_1+\cdots+\mu_k\,|\, \mu_1\in \Pi(\lambda_1),\dotsc,\mu_k\in \Pi(\lambda_k)\}.
\]
Consequently, using Proposition \ref{prop5} and Lemma \ref{lemma6} we obtain
\begin{align*}
\max_{\lambda\in B_Y(n)} d(\lambda) &\le 
\max_{\lambda\in B_Y(n)} 
(q^{-1}-q)^{-s/2} q^{-\la \rho|2\rho\ra } 
q^{- \la \lambda|2\rho\ra }\\
&\le 
(q^{-1}-q)^{-s/2} q^{- \la \rho|2\rho\ra } 
\bigl( \,
\max_{\mu\in\bigcup_{\lambda\in Y}\Pi(\lambda)}
q^{- \la \mu|2 \rho\ra }\bigr)^n
\\&=
(q^{-1}-q)^{-s/2} q^{- \la \rho|2 \rho\ra } 
\bigl( \,
\max_{\lambda\in Y}
q^{- \la \lambda|2 \rho\ra }\bigr)^n.
\end{align*}
Since $n\mapsto \# B_Y(n)$ has polynomial growth (Lemma \ref{lemma4}), it follows that
\begin{align*}
\sqrt[n]{|B_Y(n)|}&\le 
(\# B_Y(n))^{1/n}  
\max_{\lambda\in B_Y(n)} d(\lambda)^{2/n}\\
&\le 
(\# B_Y(n))^{1/n}  
(q^{-1}-q)^{-s/n}q^{-4\la\rho|\rho\ra/n}
\max_{\lambda\in Y} q^{-4\la\lambda|\rho\ra}
\xrightarrow[n\to\infty]{}
\max_{\lambda\in Y} q^{-4\la\lambda|\rho\ra}.
\end{align*}
On the other hand, let $\lambda_0\in Y$ be a representation which maximises the expression $\la\lambda|\rho\ra\,(\lambda\in Y)$. Since $n \lambda_0 \in \bf{P}^+$ is a subrepresentation of $\lambda_0\otimes\cdots\otimes \lambda_0$ ($n$-fold tensor product), we have $n\lambda_0\in B_Y(n)$ and using Proposition \ref{prop5} we obtain
\[
\sqrt[n]{|B_Y(n)|}\ge 
\sqrt[n]{|\{n\lambda_0\}|}=
d(n\lambda_0)^{2/n} \ge 
\bigl(q^{-\la n\lambda_0|2 \rho\ra }\bigr)^{2/n}=
 q^{-4\la \lambda_0|\rho\ra }.
\]
\end{proof}

Using the fact that $\la\varpi_i|\rho\ra \ge 0\,(1\le i \le r)$ and any non-zero $\lambda\in \bf{P}^+$ is of the form $\lambda=\sum_{i=1}^{r} \lambda_i \varpi_i$ with $\lambda_i\in \ZZ_+\,(1 \le i \le r)$ and $\sum_{i=1}^{r}\lambda_i\ge 1$, we obtain as a corollary a lower bound for $\omega(R(G_q))$.

\begin{corollary}\label{cor1}
Let $G$ be a simply connected, semisimple compact Lie group and $0<q<1$. Then
\[
\omega(R(G_q))  \ge \min_{1 \le i \le r } q^{-4\la\varpi_i |\rho\ra}.
\]
\end{corollary}

%We will show that the above lower bound in the case when $G$ is simple and not of type $B_N(N\ge 4),D_N(N\ge 5)$ (see the next paragraph) is  in fact an equality.
This lower bound will give us a way of calculating the uniform exponential growth rate of $R(G_q)$ (see Remark \ref{rem1} and Remark \ref{rem3}).

Recall that one says that $G$ is\footnote{Bourbaki \cite{Bourbaki1-3} calls $G$ \emph{almost simple}.} \emph{simple}, if its (complexified) Lie algebra $\mf{g}$ is simple. An arbitrary compact, simply connected, semisimple Lie group $G$ can be written as a product $G\simeq \prod_{a=1}^{l} G_{a}$, where each factor $G_a$ is a compact, simply connected, simple Lie group (\cite[Proposition III.9.28]{Bourbaki1-3}). It is well known that a similar decomposition holds as well for $q$-deformations. 

\begin{proposition}\label{prop8}
If $G\simeq \prod_{a=1}^{l}G_a$ is a decomposition of $G$ into simple factors, then $G_q$ is isomorphic with $\prod_{a=1}^{l} (G_{a})_q$.
\end{proposition}

Note that one can prove this result by establishing first an analogous decomposition for $q$-deformed universal enveloping algebra $U_q \mf{g}$.

\begin{remark}\label{rem1}
Because of Proposition \ref{prop4} and Proposition \ref{prop8}, in order to compute the uniform exponential growth rate of $R(G_q)$, it is enough to consider the situation when $G$ is simple. We do this in Theorem \ref{thm1}.
\end{remark}

Simply connected, simple, compact Lie groups are completely classified: $G$ can be of classical type $A_N\,(N\ge 1)$, $B_N\,(N\ge 2)$, $C_N(N\ge 3)$, $D_N(N\ge 4)$ or of exceptional type $E_6,E_7,E_8,F_4,G_2$, see \cite[Section 11.4]{Humphreys}.

\begin{theorem}\label{thm1}
Let $0<q<1$ and $G$ be a compact, simply connected, simple Lie group. We have the following results:
\begin{itemize}
\item type $A_N$: $G=\SU(N+1)\,(N\ge 1)$: $\omega(R(G_q))=q^{-2N}$;
\item type $B_N$: $G=\oon{Spin}(2N+1)\,(N\ge 2)$: $\omega(R(G_q))=q^{-2N^2}$;
\item type $C_N$: $G=\oon{Sp}(2N)\,(N\ge 3)$: $\omega(R(G_q))=q^{-4N}$;
\item type $D_N$: $G=\oon{Spin}(2N)\,(N\ge 4)$: $\omega(R(G_q))=q^{-N(N-1)}$.
\end{itemize}
In exceptional cases we have the following results:
\begin{itemize}
\item type $E_6$: $\omega(R(G_q))=q^{-32}$;
\item type $E_7$: $\omega(R(G_q))=q^{-54}$;
\item type $E_8$: $\omega(R(G_q))=q^{-116}$;
\item type $F_4$: $\omega(R(G_q))=q^{-44}$;
\item type $G_2$: $\omega(R(G_q))=q^{-20}$.
\end{itemize}
\end{theorem}

\begin{proof}
In classical types $A_N$ -- $D_N$, we use the information about fusion rules provided in \cite{KlimykSchmudgen}, in particular we fix the ordering of simple positive roots as in \cite{KlimykSchmudgen}. Klimyk-Schmudgen parametrise weights $\lambda\in \bf{P}^+$ in two ways, by $(n_i)_{i=1}^{N}$ and $(m_i)_i$, see \cite[Page 201]{KlimykSchmudgen}. Numbers $(n_i)_{i=1}^{N}$ are defined by $n_i=2\tfrac{\la\alpha_i| \lambda \ra }{\la \alpha_i|\alpha_i\ra}$, so that $\lambda=\sum_{i=1}^{N} n_i\varpi_i$. Parametrisation by $(m_i)_i$ depends on type, and we note that in type $A_N$ sequences $(m_i)_i$ and $(m_i+m)_i$, for an arbitrary $m\in\ZZ$, correspond to the same representation. We will also use the calculation of numbers $\la \varpi_i|\rho\ra$ from \cite{KrajczokSoltan}. We can copy these numbers without change, since as exhibited in \cite[Equation $(4.22)$]{KrajczokSoltan}, they depend only on the Cartan matrix $\bigl(2 \tfrac{\la\alpha_i|\alpha_j\ra}{\la\alpha_i|\alpha_i\ra}\bigr)_{i,j=1}^{r}$ and square-lengths of roots $\la\alpha_i|\alpha_i\ra$, and this data is on the nose the same in \cite{KrajczokSoltan} and in \cite{KlimykSchmudgen} (see \cite[Page 159]{KlimykSchmudgen}). Information about the action of $w_\circ$ can be taken from \cite[Page 265 -- 273]{BourbakiLie4-6}.\\
%Cartan matrices are the same on the nose = transpose of matrices from Wei-Zou \cite[Page 160]{KlimykSchmudgen}. Also square-lengths are the same = 2 for types A,D, and odd length for root with number N, Page 160

(Type $A_N$) 

For $1\le i \le N$ we have $\la \varpi_i |\rho\ra =\tfrac{1}{2}(N+1-i)i$. Hence $\min_{1 \le i \le N}\la \varpi_i|\rho\ra = \la\varpi_1|\rho\ra=\tfrac{N}{2}$. Representation $\varpi_1$ corresponds to numbers $n_i=\delta_{i,1}$ and $m_i=\delta_{i,1}$, hence by \cite[Equation $(13)$, page 210]{KlimykSchmudgen} $\varpi_1$ is generating. When $N\ge 2$, $\varpi_1$ is not self-contragradient, but $\ov{\varpi_1}=\varpi_N$. Consequently, we can take $X=\{\varpi_1,\varpi_N\}$ as a (symmetric) generating set. Proposition \ref{prop7} gives $\omega_X(R(G_q))=q^{-2N}$, which is equal to the lower bound of Corollary \ref{cor1}. Thus $\omega(R(G_q))=q^{-2N}$.\\

(Type $B_N$) 

For $1\le i < N$ we have $\la \varpi_i |\rho\ra =i(2N-i)$, and $\la\varpi_N|\rho\ra = \tfrac{N^2}{2}$.  It follows that all representations are self-contragradient, which is also stated in \cite[Page 268]{BourbakiLie4-6}. The relation between parametrisations $(n_i)_{i=1}^{N}$ and $(m_i)_{i=1}^{N}$ is as follows:
\[
n_i=m_i-m_{i+1}\;(1\le i \le N-1),\quad
n_N=2m_N.
\]
The numbers $(m_i)_{i=1}^{N}$ are either all integers, or all half-integers and satisfy $m_1\ge \cdots \ge m_N\ge 0$. We see that the fundamental weights $\varpi_j$ correspond to the following sequences $(m_i)_{i=1}^{N}$:
\begin{equation}\label{eq7}
\varpi_j\leftrightarrow m_i=\begin{cases}
1, & 1 \le i \le j,\\
0, & j<i\le N
\end{cases} \textnormal{ for }1\le j \le N-1 \quad \textnormal{and}\quad 
\varpi_N\leftrightarrow m_i=\tfrac{1}{2},\;\; 1 \le i \le N .
\end{equation}

We know from \cite[Proposition 20, page 211]{KlimykSchmudgen} that the representation $\varpi_1$ is not generating, but representations which occur as subrepresentations of some tensor power $\varpi_1 \otimes \cdots\otimes \varpi_1$ are precisely those with $m_i\in \ZZ$. It follows that any generating set $Y\subseteq \bf{P}^+$ must include $\mu=\sum_{i=1}^{N}\mu_i \varpi_i\in Y$ with $\mu_N\neq 0$. Indeed, assume otherwise and take $\mu^1,\dotsc,\mu^K\in Y$ such that $\varpi_N\subseteq \mu^1\otimes \cdots\otimes \mu^K$. By the assumption, we can find $A_k\in \NN$ such that $\mu^k\subseteq \varpi_1^{\otimes A_k}$. Hence
\[
\varpi_N\subseteq \mu^1\otimes \cdots \otimes \mu^K\subseteq \varpi_1^{\otimes A_1}\otimes\cdots \otimes\varpi_1^{\otimes A_K}
\]
which contradicts \eqref{eq7}. Proposition \ref{prop7} implies $\omega(R(G_q))\ge q^{-4 \la \varpi_N|\rho\ra }=q^{-2N^2}$. From \cite[equation (15), page 210]{KlimykSchmudgen} we deduce that $\{\varpi_1,\varpi_N\}$ is generating. In fact, this equation shows that $\varpi_N\subseteq \varpi_N\otimes \varpi_1$, hence by the Frobenius reciprocity $\varpi_1\subseteq \varpi_N\otimes \varpi_N$ and $X=\{\varpi_N\}$ is generating. It follows that $\omega(R(G_q))=\omega_{X}(R(G_q))=q^{-2N^2}$.\\

(Type $C_N$) 

For $1\le i \le N$ we have $\la \varpi_i |\rho\ra =(2N+1-i)\tfrac{i}{2}$. As in the previous case, all representations are self-contragredient. We have $\min_{1\le i \le N}\la \varpi_i|\rho\ra = \la \varpi_1|\rho\ra =N$. As the representation $\varpi_1$ corresponds to $m_i=\delta_{i,1}$, it is generating (\cite[Proposition 20, page 201]{KlimykSchmudgen}). For the generating set $X=\{\varpi_1\}$ we obtain $\omega(R(G_q))=\omega_X(R(G_q))=q^{-4N}$.\\

(Type $D_N$) 

For $1\le i \le N-2$ we have $\la \varpi_i |\rho\ra = (2N-i-1)\tfrac{i}{2}$, and $\la \varpi_{N-1}|\rho\ra = \la \varpi_N|\rho\ra=(N-1)\tfrac{N}{4}$. In this case, the correspondence between the numbers $(n_i)_{i=1}^{N}$ and $(m_i)_{i=1}^{N}$ is given by
\[
n_i=m_i-m_{i+1}\;(1\le i \le N-1),\quad
n_N=m_{N-1} + m_{N}.
\]
Furthermore, the numbers $(m_i)_{i=1}^{N}$ are either all integers or all half-integers and satisfy $m_1\ge \cdots \ge m_{N-1}\ge |m_{N}|$. The fundamental weights $\varpi_j$ correspond to $(m_i)_{i=1}^{N}$ as follows:
\begin{equation}\begin{split}\label{eq8}
\varpi_j&\leftrightarrow
m_i=\begin{cases} 
1, & 1\le i \le j,\\
0, & j<i\le N,
\end{cases}
\quad \textnormal{for}\quad 1\le j \le N-2,\\
\varpi_{N-1}&\leftrightarrow m_i=
\begin{cases} \tfrac{1}{2}, & 1\le i \le N-1,\\
-\tfrac{1}{2}, & i=N,
\end{cases}\\
\varpi_N &\leftrightarrow m_i=\tfrac{1}{2},\quad 1\le i\le N.
\end{split}\end{equation}

The representation $\varpi_1$ is not generating, but representations which appear as subrepresentations of some tensor power of $\varpi_1$ are precisely these $\lambda\in \bf{P}^+$ with $m_i\in \ZZ$ (\cite[Proposition 20, page 210]{KlimykSchmudgen}). In particular, $\varpi_i$ with $1\le i \le N-2$ appear as subrepresentations of tensor powers of $\varpi_1$, and as in type $B_N$ we see that any generating set $Y\subseteq \bf{P}^+$ must contain $\mu=\sum_{i=1}^{N}\mu_i \varpi_i\in Y$ with $\mu_{N-1}\neq 0$ or $\mu_N\neq 0$. From Corollary \ref{cor1} follows that $\omega(R(G_q))\ge q^{-4 \la \varpi_N|\rho\ra}=q^{-N(N-1)}$. The contents of \cite[Page272]{BourbakiLie4-6} and the rule \cite[Equation (14) page 210]{KlimykSchmudgen} can be used to deduce that $X=\{\varpi_1,\varpi_{N-1},\varpi_N\}$ is generating (in particular it is symmetric). Since $\la \varpi_1|\rho\ra=N-1\le  (N-1)\tfrac{N}{4}=\la\varpi_N|\rho\ra $ we obtain $\omega_X(R(G_q))=q^{-N(N-1)}$ and the claim follows.\\
%We have to be careful when taking information about w_0 from BourbakiLie4-6, as it is not obvious how omega_i translate between realisations. 
%for N even w_0=-1, so every representation is self contragradient -> {\varpi_1,\varpi_N} is generating, s.a.
%for N odd, as w_0=/=-1 using formula for <varpi_i|\rho\ra we have that \ov{\varpi_N}=\varpi_{N-1} and the rest of \varpi_i is s.a. so {\varpi_1,\varpi_{N-1},\varpi_N}$ is generating s.a. 
%it is not clear if we can get rid of \varpi_1 and still have generating set, but we don't need this

In exceptional types, we calculate decomposition of tensor products using \emph{Sage} \cite{sagemath}. According to the documentation \cite{SageDocumentation}, Sage's realisation in types $E-G$ is in agreement with \cite{BourbakiLie4-6}. Let use note however that while in types $E_6,E_7,E_8,F_4$ Sage makes the same choice of positive roots as \cite{BourbakiLie4-6}, in type $G_2$ the choice is different. In this case, we will follow Sage's convention. Furthermore, in type $F_4$ we have to rescale inner product, so that short roots have square-length equal to $2$.
\\
%According to https://wiki.sagemath.org/Publications_using_SageMath we might need to figure out a proper way of citing sage

%simple roots and fundamental weights of E6 agree with Bourbaki and normalisation is ok
(Type $E_6$) 

We directly calculate that $(\la \varpi_1|\rho\ra,\dotsc,\la\varpi_6|\rho\ra)=(8,11,15,21,15,8)$. Consequently we have $\min_{1\le i \le 6} \la \varpi_i|\rho\ra =\la \varpi_1|\rho\ra = \la \varpi_6|\rho\ra =8$. It follows from \cite[page 276]{BourbakiLie4-6} that $\ov{\varpi_1}=\varpi_6$. We claim that $X=\{\varpi_1,\varpi_6\}$ is generating. Indeed, since
\[\begin{split}
\varpi_1\otimes \varpi_1&=\varpi_3\oplus \varpi_6 \oplus (2\varpi_1),\quad 
\varpi_6\otimes \varpi_6=\varpi_1\oplus \varpi_5 \oplus (2\varpi_6),\quad 
\varpi_2,\varpi_4\subseteq \varpi_1\otimes \varpi_3,
\end{split}\]
we see that every $\varpi_i\,(1\le i \le 6)$ is a subrepresentation of some tensor product of representations from $X$. It follows that $X$ is generating. Consequently $\omega(R(G_q))=\omega_{X}(R(G_q))=q^{-32}$.\\

%simple roots and fundamental weights of E7 agree with Bourbaki and normalisation is ok
(Type $E_7$) 

We have $(\la\varpi_1|\rho\ra,\dotsc,\la\varpi_7|\rho\ra)=(17,\tfrac{49}{2},33,48,\tfrac{75}{2},26,\tfrac{27}{2})$. Thus $\min_{1\le i \le 7}\la \varpi_i|\rho\ra = \la \varpi_7|\rho\ra = \tfrac{27}{2}$. It follows that each $\varpi_i$ is self-contragradient (which has been also recorded in \cite[Page 281]{BourbakiLie4-6}). We claim that $X=\{\varpi_7\}$ is generating.  Using Sage we find
\[\begin{split}
\varpi_1,\varpi_6 &\subseteq \varpi_7\otimes\varpi_7,\quad 
\varpi_2 \subseteq \varpi_1\otimes \varpi_7,\quad 
\varpi_3 \subseteq \varpi_1\otimes \varpi_6\\
\varpi_4&\subseteq \varpi_1\otimes \varpi_3,\quad
\varpi_5 \subseteq \varpi_4\otimes \varpi_7.
\end{split}\]
Consequently $\omega(R(G_q))=\omega_{X}(R(G_q))=q^{-54}$.\\

%simple roots and fundamental weights of E8 agree with Bourbaki and normalisation is ok
(Type $E_8$) 

We have $(\la \varpi_1|\rho\ra ,\dotsc,\la \varpi_8|\rho\ra)=(46,68,91,135,110,84,57,29)$, so that $\min_{1\le i \le 8}\la \varpi_i|\rho\ra = \la \varpi_8|\rho\ra =29$. We again see that each $\varpi_i$ is self-contragradient. We claim that $X=\{\varpi_8\}$ is generating, which follows from
\[
\varpi_1,\varpi_7\subseteq \varpi_8\otimes\varpi_8,\quad
\varpi_2, \varpi_3, \varpi_6 \subseteq \varpi_1\otimes \varpi_7,\quad
\varpi_4,\varpi_5  \subseteq \varpi_3\otimes\varpi_6.
\]
Consequently $\omega(R(G_q))=\omega_{X}(R(G_q))=q^{-116}$.\\

%simple roots and fundamental weights of F4 agree with Bourbaki but we need to multiply S-B by *2 to get correct normalisation, e.g. ||alpha_3||^2=1
(Type $F_4$) 

In this case we have to rescale the inner product (multiply by $2$) to match it with our conventions. We have $(\la\varpi_1|\rho\ra,\dotsc,\la\varpi_4|\rho\ra )=(16,30,21,11)$. Consequently $\min_{1\le i \le 4}\la\varpi_i|\rho\ra =\la\varpi_4|\rho\ra=11$ and each $\varpi_i$ is self-contragradient. It follows from
\[
\varpi_1, \varpi_3 \subseteq \varpi_4\otimes\varpi_4,\quad 
\varpi_2\subseteq \varpi_1\otimes\varpi_1
\]
that $X=\{\varpi_4\}$ is generating. Thus $\omega(R(G_q))=\omega_{X}(R(G_q))=q^{-44}$.\\

%Sage uses different choice of simple roots than Bourbaki! Root system is the same. But this amounts to acting on "everything" with Weyl group element, which doesn't change important properties. Normalisation is ok
(Type $G_2$) 

In this case we follow the conventions of Sage when it comes to the choice and labelling of simple roots. We have $(\la \varpi_1|\rho\ra,\la\varpi_2|\rho\ra)=(5,9)$. As $\varpi_2\subseteq \varpi_1\otimes\varpi_1$ we conclude that the set $X=\{\varpi_1\}$ is symmetric, generating and $\omega(R(G_q))=\omega_{X}(R(G_q))=q^{-20}$.

\end{proof} 

\begin{remark}
	Note once again that the computations with Sage were only used to establish decompositions of tensor products of representations corresponding to fundamental weights, so that we could determine whether certain sets of representations are generating.
\end{remark}

\section{The uniform exponential growth rate for the fusion algebra of $U_F^+$}\label{sec:UF}

Let $N\ge 2, F\in GL_N(\CC)$ and $U_F^+$ be the associated free unitary quantum group (see \cite{BanicaUnitary} and \cite[Section 6.4]{Timmermann}). Since $U_{\lambda F}^+=U_F^+$ for any $\lambda>0$, we may and will assume that the matrix $F$ satisfies $\Tr(F^*F)=\Tr((F^*F)^{-1})$. In this section we will show that $R(U_F^+)$ has uniform exponential growth. We will also calculate the value of $\omega(R(U_F^+))$ and show that it is attained for the canonical generating set. In particular, we recover thus some of the results of \cite{BanicaVergnioux} for $U_N^+$. Finally, we will describe the precise asymptotics of $\omega(R(U_F^+))$ with respect to  the quantum dimension of the fundamental representation of $U_F^+$ tending to infinity.

Let us recall basic facts regarding  the representation theory of $U_F^+$, established in \cite{BanicaUnitary}. The set $\Irr(U_F^+)$ can be identified with the free product of monoids $\ZZ_+\star \ZZ_+=\la \alpha,\ov\alpha\ra $ in such a way that $\alpha$ corresponds to the fundamental representation, the empty word $e$ corresponds to the trivial representation and $w\mapsto \ov{w}$ to taking the conjugate representation. The fusion rules of $U_F^+$ are given by
\begin{equation}\label{eq19}
x\otimes y = \bigoplus_{\overset{a,b,c\in \ZZ_+\star\ZZ_+\colon}{x=ac, \,\ov{c}b=y}}ab,\qquad x,y\in \ZZ_+\star\ZZ_+.
\end{equation}

For $n\in\NN,\eps\in\{+1,-1\}$ denote $w^n_\eps=\alpha^{\eps}\alpha^{-\eps}\cdots$ ($n$ letters), where $\alpha^{+1}=\alpha, \alpha^{-1}=\ov\alpha$; set also $w^0_\eps=e$. Let $s(w),t(w)$ denote respectively the \emph{source} and the \emph{target} of a word $w \in\ZZ_+\star\ZZ_+$, so e.g.~$t(w^n_\eps)=\alpha^\eps$ for any $n \in \NN$, whereas $s(w^n_\eps)=\alpha^\eps$ if $n$ is odd and $s(w^n_\eps)=\alpha^{-\eps}$ if $n$ is even. Any $\gamma\in \Irr(U_F^+)\setminus\{e\}$ is of the form $\gamma=w^{k_1}_{\eps_{1}}\cdots w^{k_m}_{\eps_m}=
w^{k_1}_{\eps_{1}}\otimes\cdots \otimes w^{k_m}_{\eps_m}$ for the unique $m\in\NN, k_i\in \NN$ and $\eps_i\in \{+1,-1\}$ such that $s(w^{k_i}_{\eps_i})=t(w^{k_{i+1}}_{\eps_{i+1}})$. As usual, let $d\colon\ZZ_+\star \ZZ_+ \to \RR_+$ denote the quantum dimension function. We have
\begin{equation}\label{eq20}
d(w^k_\eps)=[k+1]_q,\qquad k\in \ZZ_+,\eps\in\{+1,-1\}
\end{equation}
where $0<q\le 1$ is defined by $d(\alpha)=q+q^{-1}$ (see \cite[Lemma 4.12, Lemma 4.13]{KrajczokWasilewski}). Since we assume matrix $F$ satisfies the normalisation condition $\Tr(F^*F)=\Tr((F^*F)^{-1})$, we have $\uprho_\alpha=(F^*F)^{t}$ \cite[Example 1.4.2]{NeshveyevTuset} and $q$ is determined by $q^{-1}+q=\Tr(F^*F)$.\\
%I'm changing notation to w^n_\eps, because it is convenient later
%so far no assumption on F, U_2^+ is allowe OK. Banica-Vergnioux say that it also should have exponential growth

Our first aim is to show that $\omega(R(U_F^+))=\omega_{X_{can}}(R(U_F^+))$, where $X_{can}=\{\alpha,\ov\alpha\}$. To this end we need to establish several lemmas.

\begin{lemma}\label{lemma10}
For $m \in \NN$, $k,k_1,\dotsc,k_m\in \NN$ such that $k=k_1+\cdots +k_m$ and any $\eps,\eps_1,\dotsc,\eps_m\in\{+1,-1\}$ we have $d(w^k_\eps) \le d(w^{k_1}_{\eps_1})\cdots 
d(w^{k_m}_{\eps_m})$. In particular $d(w^k_\eps)\le d(\alpha)^k$.
\end{lemma}

\begin{proof}
Observe that the number $d(w^{k_1}_{\eps_1})\cdots 
d(w^{k_m}_{\eps_m})$ does not depend on $\eps_i$'s. Hence we can assume without loss of generality that $\eps_1=\eps$ and $s(w^{k_i}_{\eps_i})\neq  t(w^{k_{i+1}}_{\eps_{i+1}})$ for $1\le i\le m-1$. Then
\[
w^k_\eps=
(w^{k_1}_{\eps_1})\cdots(w^{k_m}_{\eps_m})\subseteq w^{k_1}_{\eps_1}\otimes\cdots
\otimes w^{k_m}_{\eps_m}
\]
by \eqref{eq19} and the first claim follows. The second claim holds, as we can take $k_1,\dotsc,k_m=1$.
\end{proof}

%Let $X\subseteq \Irr(U_F^+)$ be a finite generating set. 
For $x\in \Irr(U_F^+),k\in\NN$ set $\widetilde{w}(x)^k_{+1}=x \ov{x} \cdots $ ($k$ times) and $\widetilde{w}(x)^k_{-1}=\ov{x} x \cdots$ ($k$ times). In what follows, it will be convenient to write also $x^{+1}=x, x^{-1}=\ov{x}$. The next two technical lemmas show (intuitively) that after the substitution $\alpha\mapsto x, \ov\alpha\mapsto \ov x$, the quantum dimension must grow. They are crucial for the proof of Proposition \ref{prop13}.

\begin{lemma}\label{lemma11}
Let $x\in \Irr(U_F^+)$, $x \neq e$. % (so that $d(x)\ge d(\alpha)$). 
For $k\in\NN,\eps\in \{+1,-1\}$ we have $d(\widetilde{w}(x)^k_\eps) \ge d(w^k_\eps)=[k+1]_q$.
\end{lemma}

\begin{proof}
Assume first that $\eps=+1$. %Since $d(x)\ge d(\alpha)$, we have $x\neq e$ and w
We can write $x=w^{k_1}_{\eps_1}\cdots w^{k_m}_{\eps_m}$ for some $m\in\NN, k_i\ge 1,\eps_i\in \{+1,-1\}$ such that $s(w^{k_i}_{\eps_i})=t(w^{k_{i+1}}_{\eps_{i+1}})\,(1\le i \le m-1)$. Denote $\alpha^{\eta_i}=s(w^{k_i}_{\eps_i})$. The claim is trivial if $k=1$, hence assume $k\ge 2$. Assume first that $m=1$. Then
\[
\widetilde{w}(x)^k_{+1}=
\underbrace{x \ov{x}\cdots x^{(-1)^{k+1}}}_{k}=
\underbrace{(w^{k_1}_{\eps_1})\,
(w^{k_1}_{-\eta_1})\,
(w^{k_1}_{\eps_1})\,\cdots\,}_{k}
\]
Since $s(w^{k_1}_{\eps_1})=\alpha^{\eta_1}$ and $s(w^{k_1}_{-\eta_1})=\alpha^{-\eps_1}$, we have
\[
\widetilde{w}(x)^k_{+1}=
w_{\eps_1}^{k  k_1}\quad\Rightarrow\quad 
d(\widetilde{w}(x)^k_{+1})=
[k k_1+1]_q\ge [k+1]_q.
\]
 If $m\ge 2$ then
\[\begin{split}
\widetilde{w}(x)^k_{+1}&=
\underbrace{x \ov{x}\cdots x^{(-1)^{k+1}}}_{k}=
\underbrace{(w^{k_1}_{\eps_1}\cdots w^{k_m}_{\eps_{m}})\,
(w^{k_m}_{-\eta_m} \cdots w^{k_1}_{-\eta_1})\,(w^{k_1}_{\eps_1}\cdots w^{k_m}_{\eps_{m}})\,\cdots\,
}_{k}.
\end{split}\]
Between the brackets we see neighbouring elements $w^{k_m}_{\eps_m},w^{k_m}_{-\eta_m}$ and $w^{k_1}_{-\eta_1},w^{k_1}_{\eps_1}$ which glue to $w^{2 k_m}_{\eps_m}$ and $w^{2 k_1 }_{-\eta_1}$ (there are $k-1$ such gluings). Hence (assuming $k$ is even) we obtain
%we need to make parity choice, as the last term is important
\[
\widetilde{w}(x)^k_{+1}=
w^{k_1}_{\eps_1}
f w^{2k_m}_{\eps_m} \ov{f}
w^{2 k_1 }_{-\eta_1} \cdots 
w^{2 k_1}_{-\eta_1} f
w^{2 k_m}_{\eps_m} \ov{f} w^{k_1}_{-\eta_1}
\]
where $f=w^{k_2}_{\eps_2} \cdots w^{k_{m-1}}_{\eps_{m-1}}\in \ZZ_+\star \ZZ_+$ if $m\ge 3$ or $f=e$ if $m=2$. By analysing \eqref{eq19} we see that in fact
\[
\widetilde{w}(x)^k_{+1}=
w^{k_1}_{\eps_1}\otimes
f \otimes w^{2k_m}_{\eps_m} \otimes \ov{f}
\otimes w^{2 k_1 }_{-\eta_1} \otimes\cdots \otimes
w^{2 k_1}_{-\eta_1}\otimes f\otimes
w^{2 k_m}_{\eps_m}\otimes \ov{f}\otimes w^{k_1}_{-\eta_1}.
\]
Thus using Lemma \ref{lemma10} we deduce that 
\[\begin{split}
d(\widetilde{w}(x)^k_{+1})&\ge 
d(
w^{k_1}_{\eps_1}\otimes
 w^{2k_m}_{\eps_m}\otimes
w^{2 k_1 }_{-\eta_1} \otimes\cdots \otimes
w^{2 k_1}_{-\eta_1}\otimes
w^{2 k_m}_{\eps_m} \otimes w^{k_1}_{-\eta_1})\\
&\ge 
d(w^{k_1 + 2k_m \frac{k}{2} + 2k_1 (\frac{k}{2}-1)+k_1}_{\eps_1})=
[k (k_m+k_1)+1]_q \ge 
[k+1]_q.
\end{split}\]
The proofs of other cases ($k$ odd, $m\ge 2$ and $\eps=-1$) are virtually the same.
\end{proof}

\begin{lemma}\label{lemma12}
Take $x\in \Irr(U_F^+)$ which is not of the form $w^{2k}_{\eps}$ with $k\in\ZZ_+,\eps\in\{+1,-1\}$. For $w^{l_1}_{\eps_1}\cdots w^{l_p}_{\eps_p}\in \ZZ_+\star \ZZ_+$ with $p, l_1,\dotsc,l_p\ge 1$ and $\eps_1,\dotsc,\eps_p\in\{+1,-1\}$ such that $s(w^{l_i}_{\eps_i})=t(w^{l_{i+1}}_{\eps_{i+1}})\,(1\le i \le p-1)$ we have
\[
d(\widetilde{w}(x)^{l_1}_{\eps_1}\cdots 
\widetilde{w}(x)^{l_p}_{\eps_p} ) \ge 
d(w^{l_1}_{\eps_1}\cdots 
w^{l_p}_{\eps_p} ).
\]
\end{lemma}

\begin{proof}
Lemma \ref{lemma11} establishes Lemma \ref{lemma12} in case $p=1$, so assume $p\ge 2$.

Write $x=w^{k_1}_{\delta_1}\cdots w^{k_m}_{\delta_m}$ with $m,k_1,\dotsc,k_m\ge 1$ and $\delta_i\in \{+1,-1\}$ such that $s(w^{k_i}_{\delta_i})=t(w^{k_{i+1}}_{\delta_{i+1}})$ for $1\le i \le m-1$. Further for $1\le i \le m-1$  set $\eta_i\in\{+1, -1\}$ so that  $\alpha^{\eta_i}=s(w^{k_i}_{\delta_{i}})$. We then have
\begin{equation}\label{eq21}
\widetilde{w}(x)^{l_1}_{\eps_1}\cdots
\widetilde{w}(x)^{l_p}_{\eps_p}=
( \underbrace{x^{\eps_1}x^{-\eps_1} \cdots 
x^{(-1)^{l_1+1} \eps_1}}_{l_1})
( \underbrace{x^{\eps_2}x^{-\eps_2} \cdots 
x^{(-1)^{l_2+1} \eps_2}}_{l_2})
\cdots 
( \underbrace{x^{\eps_p}x^{-\eps_p} \cdots 
x^{(-1)^{l_p+1} \eps_p}}_{l_p}).
\end{equation}
As $(-1)^{l_{i}+1}\eps_i = \eps_{i+1}$,  between the brackets in \eqref{eq21} we see expressions of the form  ``$\dotsc x ) ( x \dotsc$'' or ``$\dotsc \ov{x} ) (\ov{x}\dotsc$''. Observe that since $x\notin\{ w^{2k}_\eps\mid k\in\ZZ_+,\eps\in\{+1,-1\}\}$, it is enough to consider two cases: $\eta_m=\delta_1$ and $\eta_m\neq \delta_1,m\ge 2$.\\

\emph{Case 1: $\eta_m=\delta_1$}. In this situation 
\begin{align*}
\widetilde{w}(x)^{l_1}_{\eps_1}\cdots&
\widetilde{w}(x)^{l_p}_{\eps_p} =\\
&=
( \underbrace{x^{\eps_1}x^{-\eps_1} \cdots 
x^{(-1)^{l_1+1} \eps_1}}_{l_1})\otimes
( \underbrace{x^{\eps_2}x^{-\eps_2} \cdots 
x^{(-1)^{l_2+1} \eps_2}}_{l_2})\otimes
\cdots \otimes
( \underbrace{x^{\eps_p}x^{-\eps_p} \cdots 
x^{(-1)^{l_p+1} \eps_p}}_{l_p})\\
&=
\widetilde{w}(x)^{l_1}_{\eps_1}\otimes\cdots\otimes
\widetilde{w}(x)^{l_p}_{\eps_p}
\end{align*}
hence the claim follows from Lemma \ref{lemma11} and multiplicativity of the dimension function.\\

\emph{Case 2: $\eta_m\neq \delta_1, m\ge 2$}.
In total, in \eqref{eq21} we have
% $l_1+\dotsc+l_p$ symbols $x^{\pm 1}$, each contains $m$ symbols $w^{\bullet}_{\bullet}$. There are 
$l_1+\dotsc+l_p -1$ spaces between two $x^{\pm 1}$'s, of which  $p-1$ lie between the brackets and $l_1+\dotsc+l_p-p$ lie inside the brackets. In each of these spaces, a gluing between two $w^\bullet_\bullet$'s happens.

Inside the brackets we see the situation of the kind
\[
x \ov{x}=
w^{k_1}_{\delta_1}\cdots w^{k_m}_{\delta_m}
w^{k_m}_{-\eta_m} \cdots
w^{k_1}_{-\eta_1}=
w^{k_1}_{\delta_1}\cdots w^{k_{m-1}}_{\delta_{m-1}} w^{2 k_m}_{\delta_m}
w^{k_{m-1}}_{-\eta_{m-1}} \cdots
w^{k_1}_{-\eta_1}
\]
or
\[
\ov{x} x =
w^{k_m}_{-\eta_m} \cdots
w^{k_1}_{-\eta_1}
w^{k_1}_{\delta_1}\cdots w^{k_m}_{\delta_m}=
w^{k_m}_{-\eta_m} \cdots
w^{k_2}_{-\eta_2}
w^{2 k_1}_{-\eta_1}
w^{k_2}_{\delta_2}\cdots w^{k_m}_{\delta_m}.
\]
 Between the brackets we have the expressions
\[
xx=w^{k_1}_{\delta_1}\cdots w^{k_m}_{\delta_m}\,
w^{k_1}_{\delta_1}\cdots w^{k_m}_{\delta_m}=
w^{k_1}_{\delta_1}\cdots w^{k_{m-1}}_{\delta_{m-1}}w^{ k_m+k_1}_{\delta_m}\,
w^{k_2}_{\delta_2} \cdots w^{k_m}_{\delta_m}
\]
or 
\[
\ov{x}\, \ov{x} = 
w^{k_m}_{-\eta_m} \cdots
w^{k_1}_{-\eta_1}
w^{k_m}_{-\eta_m} \cdots
w^{k_1}_{-\eta_1}=
w^{k_m}_{-\eta_m} \cdots w^{k_2}_{-\eta_2}
w^{k_1+k_m }_{-\eta_1}
w^{k_{m-1}}_{-\eta_{m-1}} \cdots 
w^{k_1}_{-\eta_1}.
\]
Recall that $d(w^n_{+1})=d(w^n_{-1})\,(n\in\ZZ_+)$. From the above analysis  we conclude, simply counting the elements of the form $xx,  \ov{x} x$, etc.\ inside the expression \eqref{eq21}, that 
\begin{align*}
d( 
&\widetilde{w}(x)^{l_1}_{\eps_1}\cdots
\widetilde{w}(x)^{l_p}_{\eps_p})=
d(w^{h(\eps_1)}_{+1}) \bigl( \prod_{i=2}^{m-1}\, d(w^{k_i}_{\delta_i})\bigr)^{l_1+\cdots + l_p}\,
d(w^{2k_m}_{\delta_m})^{f(\eps_1,l_1)+\cdots+
f(\eps_p,l_p)}
\\
&
d(w^{2k_1}_{-\eta_1})^{g(\eps_1,l_1)+\cdots+
	g(\eps_p,l_p)}
d(w^{k_m+k_1}_{\delta_m})^{ \#\{2 \le i\le p\mid \eps_i=+1\}}
d(w^{k_1+k_m}_{-\eta_1})^{ \#\{2 \le i \le p\mid \eps_i=-1\}}d(w^{h((-1)^{l_p}\eps_p)}_{+1}),
\end{align*}
where the empty product is interpreted as $1$, $h(+1)=k_1$, $h(-1)=k_m$ and $f(+1,l)=\lfloor \tfrac{l}{2}\rfloor$, $f(-1,l)=\lceil \tfrac{l}{2}\rceil -1$, $g(+1, l)=\lceil \tfrac{l}{2}\rceil -1$, $g(-1,l)=\lfloor \tfrac{l}{2}\rfloor$. We can simplify the last expression as follows:
\begin{align*}
d( 
\widetilde{w}&(x)^{l_1}_{\eps_1}\cdots
\widetilde{w}(x)^{l_p}_{\eps_p}) =
[h(\eps_1)+1]_q  \bigl( \prod_{i=2}^{m-1}\, [k_i+1 ]_q\bigr)^{l_1+\cdots + l_p}\,
[2k_m+1]_q^{f(\eps_1,l_1)+\cdots+
f(\eps_p,l_p)}
\\
&\quad\quad\quad\quad\quad
\quad\quad\quad\quad\;\;\;\;\;
[2k_1+1]_q^{g(\eps_1,l_1)+\cdots+
	g(\eps_p,l_p)}
[ k_1+k_m+1]_q^{p-1} [h((-1)^{l_p}\eps_p)+1]_q.
\end{align*}
Observe that $f(\eps,l)+g(\eps,l)=l-1$ and $h(\eps)\ge \min(k_1,k_m)\ge 1$ for $\eps\in \{+1,-1\}$. Using this we can bound $d( \widetilde{w}(x)^{l_1}_{\eps_1}\cdots
\widetilde{w}(x)^{l_p}_{\eps_p})$ in the following way:
\begin{align*}
&d( 
\widetilde{w}(x)^{l_1}_{\eps_1}\cdots
\widetilde{w}(x)^{l_p}_{\eps_p})\ge 
[2]_q  
[2k_m+1]_q^{f(\eps_1,l_1)+\cdots+
f(\eps_p,l_p)}
[2k_1+1]_q^{g(\eps_1,l_1)+\cdots+
g(\eps_p,l_p)}\\
&\quad\quad\quad\quad\quad\quad
\quad\quad\quad\quad\quad\,\;\;\;\;\;\;
[ k_1+k_m+1]_q^{p-1} [2]_q\\
&=
\bigl(\prod_{i=1}^{p-1}
[2k_m+1]_q^{f(\eps_i,l_i)}
[2k_1+1]_q^{g(\eps_i,l_i)} [k_1+k_m+1]_q \bigr)
[2]_q 
[2k_m+1]_q^{f(\eps_p,l_p)}
[2k_1+1]_q^{g(\eps_p,l_p)} [2]_q\\
&\ge 
\bigl(\prod_{i=1}^{p-1}
[3]_q^{f(\eps_i,l_i)}
[3]_q^{g(\eps_i,l_i)} [3]_q \bigr)
[2]_q 
[3]_q^{f(\eps_p,l_p)}
[3]_q^{g(\eps_p,l_p)} [2]_q=
\bigl(\prod_{i=1}^{p-1}
[3]_q^{l_i} \bigr)
[2]_q 
[3]_q^{l_p-1 } [2]_q\\
&\ge 
[2]_q^{l_1+\cdots + l_p}=
d(\alpha)^{l_1+\cdots + l_p}\ge 
d(w^{l_1}_{\eps_1})\cdots d(w^{l_p}_{\eps_p})= 
d(w^{l_1}_{\eps_1} \cdots w^{l_p}_{\eps_p}),
\end{align*}
as claimed.
\end{proof}

Let us now consider the canonical generating set $X_{can}=\{\alpha,\ov\alpha\}$. The following elementary lemma can be easily established using the fusion rules \eqref{eq19}.

\begin{lemma}\label{lemma13}
If $\gamma=w^{k_1}_{\eps_1}\cdots w^{k_m}_{\eps_m}\in \Irr(U_F^+)\setminus\{e\}$ with $m,k_1,\dotsc,k_m\in \NN$ and $\eps_1,\dotsc,\eps_m\in \{+1,-1\}$ such that $s(w^{k_i}_{\eps_i})=t(w^{k_{i+1}}_{\eps_{i+1}})(1\le i \le m-1)$, then $\ell_{X_{can}}(\gamma)=k_1+\cdots+k_m$.
\end{lemma}

As a corollary, we obtain a conclusion that an arbitrary generating set $X$ contains an element $x\neq \ov{x}$ of the form considered in Lemma \ref{lemma12}.

\begin{lemma}\label{lemma14}\noindent
\begin{enumerate}
\item Take $x\in \Irr(U_F^+)$. If $\ell_{X_{can}}(x)\in 2\ZZ_++1$, then $\ov{x}\neq x$.
\item Let $X\subseteq \Irr(U_F^+)$ be a finite generating set. Then there is $x\in X$ with odd length $\ell_{X_{can}}(x)$.
% In particular, there is $x\in X$ is not of the form $w^{2k}_\eps$ with $k\in\ZZ_+,\eps\in\{+1,-1\}$.
\end{enumerate}
\end{lemma}

\begin{proof}
$(1)$ Take $ x$ with $l=\ell_{X_{can}}(x)\in 2\ZZ_+ +1$. We can write (in a unique way) $x=\alpha^{\eps_1}\cdots \alpha^{\eps_{-1+(l+1)/2}} \alpha^{\eps_{(l+1)/2}}\alpha^{\eps_{1+(l+1)/2}}\cdots \alpha^{\eps_l}$, where $\eps_1,\dotsc,\eps_l\in \{+1,-1\}$ and $\alpha^{+1}=\alpha,\alpha^{-1}=\ov\alpha$. Then $\ov{x}=\alpha^{-\eps_l}\cdots \alpha^{-\eps_{1+(l+1)/2}} \alpha^{-\eps_{(l+1)/2}}\alpha^{-\eps_{-1+(l+1)/2}}\cdots \alpha^{-\eps_1}$ and we see by looking at the middle letter that $\ov{x}\neq x$.

$(2)$ Fusion rules \eqref{eq19} show that if $x_1,x_2,w\in \Irr(U_F^+)$, $w\subseteq x_1\otimes x_2$ and $\ell_{X_{can}}(x_1),\ell_{X_{can}}(x_2)$ are even, then $\ell_{X_{can}}(w)$ is also even. This shows that any generating set $X$ must contain an element with odd length. 
\end{proof}

We are now ready to deduce that the optimal exponential growth rate is obtained by considering the canonical generating set.

\begin{proposition}\label{prop13}
We have $\omega(R(U_F^+))=\omega_{X_{can}}(R(U_F^+))$, where $X_{can}=\{\alpha,\ov\alpha\}$.
\end{proposition}

\begin{proof}
Take an arbitrary finite generating set $X\subseteq \Irr(U_F^+)$ and fix $x\in X$ with $\ell_{X_{can}}(x)\in 2\ZZ_+ +1$. In particular, $x$ is not of the form $w^{2k}_{\eps}$ for $k\in\ZZ_+,\eps\in \{+1,-1\}$, and $\ov{x}\neq x$ (Lemma \ref{lemma14}). Fix $n\in\NN$. By Lemma \ref{lemma13} we have
\[
B_{X_{can}}(n)\setminus\{e\}=\{ w^{k_1}_{\eps_1}\cdots 
w^{k_m}_{\eps_m} \mid 
\eps_1\in\{+1,-1\}, m\in\NN, k_1,\dotsc,k_m\in\NN ,k_1+\dotsc+k_m\le n\},
\]
where the above $\eps_{i}$'s (with $i>1$) are determined by the requirement that $s(w^{k_i}_{\eps_i})=t(w^{k_{i+1}}_{\eps_{i+1}})$; we will also use a similar convention later on. Recall the symbols $\widetilde{w}(x)^{k}_{\eps}$ introduced before Lemma \ref{lemma11}. For $\eps_1\in \{+1,-1\}$, $m\in\NN,k_1, \ldots, k_m \in \NN$ such that $ k_1+\cdots+k_m\le n$ we have 
\[
\ell_X(\widetilde{w}(x)^{k_1}_{\eps_1} 
\cdots
\widetilde{w}(x)^{k_m}_{\eps_m} )\le n \quad \Rightarrow\quad 
\widetilde{w}(x)^{k_1}_{\eps_1} 
\cdots
\widetilde{w}(x)^{k_m}_{\eps_m}\in B_X(n).
\]
Observe also that the elements $\widetilde{w}(x)^{k_1}_{\eps_1} \cdots \widetilde{w}(x)^{k_m}_{\eps_m} $ are pairwise distinct. Indeed, this holds since $\ov{x}\neq x$ and the expression of any $\gamma\in\Irr(U_F^+)=\ZZ_+\star\ZZ_+$ as a word in letters $\alpha,\ov\alpha$ is unique. Using Lemma \ref{lemma12} we obtain
\[\begin{split}
|B_X(n)| &\ge |\{
\widetilde{w}(x)^{k_1}_{\eps_1}\cdots 
\widetilde{w}(x)^{k_m}_{\eps_m} 
 \mid 
\eps_1\in \{+1,-1\},m\in\NN, k_1,\dotsc,k_m\in\NN, k_1+\dotsc+k_m\le n\}|\\
&\ge 
|\{w^{k_1}_{\eps_1}\cdots 
w^{k_m}_{\eps_m} 
 \mid 
\eps_1\in \{+1,-1\},m\in\NN, k_1,\dotsc,k_m\in\NN, k_1+\dotsc+k_m\le n\}|\\
&=
|B_{X_{can}}(n)\setminus \{e\}|=
|B_{X_{can}}(n)|-1.
\end{split}\]
It follows that
\begin{align*}
\omega_X(R(U_F^+))&=
\lim_{n\to\infty} |B_X(n)|^{1/n} \ge 
\liminf_{n\to\infty} (|B_{X_{can}}(n)|-1)^{1/n}\\
&=
\liminf_{n\to\infty} 
|B_{X_{can}}(n)|^{1/n}
\bigl(1-\tfrac{1}{|B_{X_{can}}(n)|}\bigr)^{1/n}=
\omega_{X_{can}}(R(U_F^+)).
\end{align*}
\end{proof}

It remains to calculate $\omega_{X_{can}}(R(U_F^+))$. To spell out our result, let us introduce for $0<q\le 1$ the following polynomial:
\begin{equation}\label{eq22}
P_q(z)=z^3- (2q^{-2} +3+2q^2)z^2 +2(q^{-2}+1+q^2)z -2.
\end{equation}
Observe that as $P_q(1)=-2$ and $\lim_{x\to +\infty}P_q(x)=+\infty$, $P_q$ has a real root strictly greater than $1$. %Let then $r_q$ denote the largest real root of $P_q$.  
The proof of the following theorem shows in particular that the largest real root of $P_q$ is also its largest root in terms of absolute value.% and for small $q$, $P_q$ has three real roots.

\begin{theorem}\label{thm2}
Let $F\in GL_N(\CC)\,(N\ge 2)$ be such that $\Tr(F^*F)=\Tr( (F^*F)^{-1})$ and define $0<q\le 1$ via $q+q^{-1}=\Tr(F^*F)$. Then $\omega(R(U_F^+))=r_q>1$, where $r_q$ is the largest real root of the polynomial $P_q$ defined in \eqref{eq22}. In particular $R(U_F^+)$ has uniform exponential growth.
\end{theorem}

\begin{proof}
By Proposition \ref{prop13} it suffices to calculate $\omega_{X_{can}}(R(U_F^+))$. Recall that the set $\Irr(U_F^+)$ consists of the trivial representation $e$ and words of the form $w^{n_1}_{\eps_1}\cdots w^{n_m}_{\eps_m}$, where $m,n_1,\dotsc,n_m$ $\ge 1$, $\eps_1\in \{+1,-1\}$ and $\eps_i(i\ge 2)$ are chosen so that $s(w^{n_i}_{\eps_i})=t(w^{n_{i+1}}_{\eps_{i+1}})\,(1\le i \le m-1)$. Further recall from \eqref{eq20} and Lemma \ref{lemma13} that the length (with respect to $X_{can}=\{\alpha,\ov\alpha\}$) and the quantum dimension of the irreducibles as above are given by 
\begin{equation}\label{eq23}
\begin{split}
\ell_{X_{can}}(e)=0,&\quad d(e)=1,\\
\ell_{X_{can}}(w^{n_1}_{\eps_1}\cdots w^{n_m}_{\eps_m}) =n_1+\cdots+n_m,&\quad 
d(w^{n_1}_{\eps_1}\cdots w^{n_m}_{\eps_m}) =
[n_1+1]_q\cdots [n_m+1]_q.
\end{split}\end{equation}

Next, consider the compact quantum group $SU_q(2)\wh{\star} SU_q(2)$ (dual to the free product of discrete quantum groups $\wh{SU_q(2)}\star \wh{SU_q(2)}$) and recall that $\Irr(SU_q(2)\wh{\star}SU_q(2))$ consists of the trivial representation $e$ and words of the form $v^{n_1}_{\eps} v^{n_2}_{-\eps}\cdots v^{n_m}_{\pm \eps}$, where $m,n_1,\dotsc,n_m\ge 1$, $\eps\in \{+1,-1\}$, the sign $\pm$ depends on the parity of $m$ and $v_{+1}^n$ (resp.~$v_{-1}^n$) corresponds to the $n$'th irreducible representation of the first (resp.~second) copy of $SU_q(2)$ in $SU_q(2)\wh{\star}SU_q(2)$ (see \cite{Wangfree}). Let $Y=\{v^1_{+1},v^1_{-1}\}$ be a generating set of $R(SU_q(2)\wh{\star}SU_q(2))$. Then we easily see that
\begin{equation}\label{eq24}
\begin{split}
\ell_{Y}(e)=0,&\quad d(e)=1,\\
\ell_{Y}(v^{n_1}_{\eps} v^{n_2}_{-\eps}\cdots v^{n_m}_{\pm \eps}) =n_1+\cdots+n_m,&\quad 
d(v^{n_1}_{\eps} v^{n_2}_{-\eps}\cdots v^{n_m}_{\pm \eps}) =
[n_1+1]_q\cdots [n_m+1]_q.
\end{split}\end{equation}
From \eqref{eq23} and \eqref{eq24} we conclude that the map $\Phi\colon \Irr(U_F^+)\rightarrow \Irr(SU_q(2)\wh{\star}SU_q(2))$ given by 
\begin{equation}\label{eq25}
\Phi(e) =e, \;\;\;
\Phi(w^{n_1}_{\eps_1}\cdots w^{n_m}_{\eps_m})=
v^{n_1}_{\eps_1} v^{n_2}_{-\eps_1}\cdots v^{n_m}_{\pm \eps_1}
\end{equation}
is a bijection which preserves quantum dimensions and lengths with respect to $X_{can}$ and $Y$. Consequently
\[
\omega_{X_{can}}(R(U_F^+))=
\omega_Y(R(SU_q(2)\wh{\star}SU_q(2) )).
\]
Following \cite{BanicaVergnioux}, let us define functions $\mc{S}_{\GG}(z)=\sum_{n=0}^{\infty} 
\sum_{\ell_X(v)=n} d(v)^2 z^n$ and $\mc{P}_{\GG}(z)=1-\tfrac{1}{\mc{S}_{\GG}(z)}$. These functions depend naturally on the choice of generating set $X$, and we will use them below for $SU_q(2)$ with $X=\{1\}$ and for $SU_q(2)\wh{\star}SU_q(2)$ with $X=Y$. Then  \cite[Theorem 3.2]{BanicaVergnioux} implies that
\[
\mc{P}_{SU_q(2)\wh{\star}SU_q(2)}(z)=2 \mc{P}_{SU_q(2)}(z).
\]
The function $\mc{P}_{SU_q(2)}$ is easy to calculate (c.f.~\cite[Theorem 4.1]{BanicaVergnioux}):
\[
\mc{S}_{SU_q(2)}(z)=\sum_{n=0}^{\infty} [n+1]_q^2 z^n=
\tfrac{1+z}{(1-q^{-2}z )(1-q^2 z) (1-z)}.
\]
%ok also for q=1
Hence
\[
\mc{P}_{SU_q(2)}(z)=
\tfrac{1+z -(1-q^{-2}z)(1-q^2 z)(1-z) }{1+z}.
\]
It follows that
\[
\mc{S}_{SU_q(2)\wh{\star}SU_q(2)}(z)=
\tfrac{1}{1-\mc{P}_{SU_q(2)\wh{\star}SU_q(2)}(z)}=
\tfrac{1}{1-2 \mc{P}_{SU_q(2)}(z)}=\tfrac{1+z}{
1- (2q^{-2} +3+2q^2)z +2(q^{-2}+1+q^2)z^2 -2z^3}.
\]
and
\begin{equation}\label{eq26}
\mc{S}_{SU_q(2)\wh{\star}SU_q(2)}(\tfrac{1}{z})=
\sum_{n=-\infty}^{0}
|S_Y(-n)| z^n=
\tfrac{z^2(1+z)}{
z^3- (2q^{-2} +3+2q^2)z^2 +2(q^{-2}+1+q^2)z -2}.
\end{equation}

From the convergence criterion for Laurent series, we see that $\omega_{Y}(R(SU_q(2)\wh{\star}SU_q(2)))=\lim_{n\to\infty} |S_Y(n)|^{1/n}$ is equal to the largest absolute value of a root of the polynomial $P_q$
%\begin{equation}\label{eq27}
%P_q(z)=z^3- (2q^{-2} +3+2q^2)z^2 +2(q^{-2}+1+q^2)z -2
%\end{equation}
(c.f.~\eqref{eq22}). Finally Pringsheim's theorem \cite[Theorem IV.3, p.197]{AnalyticCombinatorics} implies that this value is equal to $r_q$, which was defined to be the largest real root of $P_q$.

\end{proof}

\begin{remark}
	The proof above  follows the logic of \cite{BanicaVergnioux}, but avoids the usage of the notion of the \emph{free version}. One could also provide the proof using a monoidal equivalence result \cite[Corollary 6.3]{BichonDeRijdtVaes} and then \cite[Th{\'e}or{\`e}me 1]{BanicaUnitary} and \cite[Theorem 3.3]{BanicaVergnioux}. In particular when $F=\I\in GL_N(\CC)$, the calculation of $\omega_{X_{can}}(R(U_N^+))$ recovers the calculation of the exponential growth rate for $A_u(N)$ in \cite{BanicaVergnioux}. 
\end{remark}

Although the last theorem computes the uniform exponential growth rate for $R(U_F^+)$, the form of the answer is not very explicit. Below we present an alternative approach leading to an estimate for the value of the root $r_q$ appearing above, and in particular recovering the asymptotic behaviour as $q$ tends to $0^+$ (or, in other words, when the quantum dimension of the canonical fundamental representation of $U_F^+$ goes to infinity). For unitary $F\in GL_2(\CC)$ we have $q=1$ and the root $r_1$ can be computed explicitly (giving $r_1=6.065\dotsc$); hence below we exclude this case.

\begin{proposition}\label{prop16}
	Let $F\in GL_N(\CC)\,(N\ge 2)$ be such that $\Tr(F^*F)=\Tr( (F^*F)^{-1})>2$ and define $0<q< 1$ via $q+q^{-1}=\Tr(F^*F)$.
	We have
	\[
	2q^{-2}+2 +q^2\le 
	\omega(R(U_F^+)) \le
2q^{-2}+2+q^2\, \tfrac{3-2q^2}{(1-q^2)^2}.
	\]
	In particular $\omega(R(U_F^+)) = 2 q^{-2}+ 2+\mc{O}(q^2)$ as $q\to 0^+$.
\end{proposition}

\begin{proof}
	Due to Proposition \ref{prop13}, we need to bound $\omega_{X_{can}}(R(U_F^+))$. Take $n\in\NN$ and any element $\gamma\in S_{X_{can}}(n)$. By Lemma \ref{lemma13} we can write $\gamma=w^{k_1}_{\eps_1}\cdots w^{k_m}_{\eps_m}$ with $m,k_1,\dotsc,k_m\ge 1,\eps_1\in\{+1,-1\}$, $s(w^{k_i}_{\eps_i})=t(w^{k_{i+1}}_{\eps_{i+1}})$ and $k_1+\cdots+k_m=n$. The quantum dimension of $\gamma$ equals $d(\gamma)=[k_1+1]_q\cdots [k_m+1]_q
	$. Hence 
	\[\begin{split}
&\quad\;
|S_{X_{can}}(n)|=|\{w^{k_1}_{\eps_1}\cdots w^{k_m}_{\eps_m}\mid 
			\eps_1\in \{+1,-1\},k_1+\cdots +k_m = n\}|\\
			&=
			\sum_{\eps_1\in \{+1,-1\}}\sum_{m=1}^{n} \sum_{\overset{k_1,\dotsc,k_m\ge 1,}{k_1+\cdots+k_m=n}} d(w^{k_1}_{\eps_1} \cdots w^{k_m}_{\eps_m})^2=
			2 \sum_{m=1}^{n} \sum_{\overset{k_1,\dotsc,k_m\ge 1,}{k_1+\cdots+k_m=n}} [k_1+1]_q^2 \cdots [k_m+1]_q^2.
	\end{split}\]
	Observe that 
	\begin{equation}\label{eq30}
		(1-q^4)^m 
		\tfrac{ q^{-(k_1+\cdots +k_m)-m}}{(q^{-1}-q)^{m}}\le 
		d(\gamma) \le \tfrac{ q^{-(k_1+\cdots +k_m)-m}}{(q^{-1}-q)^{m}}.
	\end{equation}
	Indeed, the upper bound holds as $[k_i+1]_q=\tfrac{q^{-k_i-1}-q^{k_i+1}}{q^{-1}-q}\le \tfrac{q^{-k_i-1}}{q^{-1}-q}$ and the lower follows from the string of equivalences 
	\[
	[k_i+1]_q = 
	\tfrac{q^{-k_i-1}-q^{k_i+1}}{q^{-1}-q} \ge 
	(1-q^4) 
	\tfrac{q^{-k_i-1}}{q^{-1}-q}\;\Leftrightarrow\; 
	q^4 q^{-k_i-1} \ge q^{k_i+1}\;\Leftrightarrow\;
	q^4 \ge q^{2k_i+2};
	\]
the last inequality holds, as $k_i\ge 1$. 
	
	Using \eqref{eq30}, we first establish an upper bound on $|S_{X_{can}}(n)|$:
	\[\begin{split}
		|S_{X_{can}}(n)| &\le 
		2 \sum_{m=1}^{n} \sum_{\overset{k_1,\dotsc,k_m\ge 1,}{k_1+\cdots+k_m=n}} 
		\tfrac{ q^{-2(k_1+\cdots +k_m)-2m}}{(q^{-1}-q)^{2m}}=
		2 q^{-2n} \sum_{m=1}^{n} 
		\tfrac{ q^{-2m}}{(q^{-1}-q)^{2m}}
		\sum_{\overset{k_1,\dotsc,k_m\ge 1,}{k_1+\cdots+k_m=n}} 
		1\\
		&=
		2 q^{-2n} \sum_{m=1}^{n} 
		\tfrac{ q^{-2m}}{(q^{-1}-q)^{2m}}
		{n-1\choose m-1}=
		2 q^{-2n} 
		\tfrac{ q^{-2}}{(q^{-1}-q)^{2}}
		\sum_{m=0}^{n-1} 
		\tfrac{ q^{-2m}}{(q^{-1}-q)^{2m}}
		{n-1\choose m}\\
		&=
		2 q^{-2n} 
		\tfrac{ q^{-2}}{(q^{-1}-q)^{2}}
		\bigl( 1+ (\tfrac{q^{-1}}{q^{-1}-q})^{2} \bigr)^{n-1}=
		2 q^{-2n} 
		\tfrac{ q^{-2}}{(q^{-1}-q)^{2}}
		\bigl( 1+ (\tfrac{1}{1-q^2})^{2} \bigr)^{n-1}.
	\end{split}\]
Since $q^{-2} \bigl( 1+\tfrac{1}{(1-q^2)^2}\bigr)=2q^{-2}+2+q^2\, \tfrac{3-2q^2}{(1-q^2)^2}$,
%https://www.wolframalpha.com/input?i=q%5E%28-2%29+%281%2B1%2F%28+%281-q%5E2%29%5E2%29%29+-+2q%5E%28-2%29+-2
the stated upper bound on $\omega_{X_{can}}(R(U_F^+))$ easily follows. A similar computation gives the lower bound:
	\[\begin{split}
		|S_{X_{can}}(n)| &\ge 
		2 \sum_{m=1}^{n} \sum_{\overset{k_1,\dotsc,k_m\ge 1,}{k_1+\cdots+k_m=n}} 
		(1-q^4)^{2m}
		\tfrac{ q^{-2(k_1+\cdots +k_m)-2m}}{(q^{-1}-q)^{2m}}\\
		&=
		2 q^{-2n} \sum_{m=1}^{n} 
		(1-q^4)^{2m}
		\tfrac{ q^{-2m}}{(q^{-1}-q)^{2m}}
		{n-1\choose m-1}\\
		&=
		2 q^{-2n}
		(1-q^4)^{2}
		\tfrac{ q^{-2}}{(q^{-1}-q)^{2}}
		\sum_{m=0}^{n-1} 
		(1-q^4)^{2m}
		\tfrac{ q^{-2m}}{(q^{-1}-q)^{2m}}
		{n-1\choose m}\\
		&=
		2 q^{-2n}
		(1-q^4)^{2}
		\tfrac{ q^{-2}}{(q^{-1}-q)^{2}}
		\bigl( 1 + 
		\tfrac{ (1-q^4)^{2}}{(1-q^{2})^{2}}
		\bigr)^{n-1}\\
		&=
		2 q^{-2n}
		(1-q^4)^{2}
		\tfrac{ q^{-2}}{(q^{-1}-q)^{2}}
		\bigl( 2+2q^2 +q^4 \bigr)^{n-1}.
	\end{split}\]
The stated asymptotics follows now easily (both the upper and lower estimate are of the form $2q^{-2} +2 + \mc{O}(q^2)$).
\end{proof}

Note that the asymptotics of $r_q$ could be also deduced from a straightforward, but rather tedious computation based on Theorem \ref{thm2}.

\begin{remark}
	Propositions \ref{prop7}, \ref{prop16} (and their proofs) show qualitatively different behaviour between the representation theory of quantum groups $U_F^+$ and $G_q$. Loosely speaking, while in the case of $G_q$ asymptotically  the size of the sphere can be attributed to a single irreducible representation, %and its quantum dimension is well approximated by $\Gamma$,
dealing with $U_F^+$ we have to consider the contribution of all the representations belonging to the sphere.% in the case of $U_F^+$ and quantum dimension is significantly larger than $\Gamma$.
\end{remark}

 \smallskip

\noindent {\bf Acknowledgments.}
J.K.\ was partially supported by FWO grant 1246624N. A.S.\ was  partially supported by the National Science Center (NCN) grant no. 2020/39/I/ST1/01566, and by a grant from the Simons Foundation. We would like to thank the Isaac Newton Institute for Mathematical Sciences, Cambridge, for support and hospitality during the programme Quantum Information, Quantum Groups and Operator Algebras, supported by EPSRC grant no. EP/R014604/1. We acknowledge a discussion with Makoto Yamashita; in particular we are grateful for the explanation regarding Proposition \ref{prop8}. Furthermore, J.K.~is grateful to Marco Matassa for a discussion concerning representation theory. Finally we thank the anonymous referee for useful comments.

\begin{comment}
\smallskip

\noindent {\bf Acknowledgments.}
J.K. was partially supported by FWO grant 1246624N. A.S.\ was  partially supported by the National Science Center (NCN) grant no. 2020/39/I/ST1/01566, and by a grant from the Simons Foundation. We would like to thank the Isaac Newton Institute for Mathematical Sciences, Cambridge, for support and hospitality during the programme Quantum Information, Quantum Groups and Operator Algebras, supported by EPSRC grant no. EP/R014604/1. We acknowledge a discussion with Makoto Yamashita; in particular we are grateful for the explanation regarding Proposition \ref{prop8}. Furthermore, J.K.~is grateful to Marco Matassa for a discussion concerning representation theory.

%We acknowledge a discussion with Marco Matassa and Makoto Yamashita; in particular we are grateful for the explanation regarding Proposition \ref{prop8}.

\end{comment}

\bibliographystyle{amsalpha}
\bibliography{bibliografia}
%\printbibliography

\end{document}